\documentclass[numbers,webpdf,imaiai]{ima-authoring-template}

\usepackage{upgreek}
\usepackage{booktabs}
\usepackage{makecell}
\usepackage{float}
\usepackage[nameinlink]{cleveref}
\hypersetup{hidelinks,
  pdftitle={A Two-level FSI Method for Generalized Hermitian Interior Eigenvalues},
  pdfauthor={ }}

\crefname{section}{section}{sections}
\crefname{subsection}{subsection}{subsections}

\theoremstyle{thmstyletwo}
\newtheorem{theorem}{Theorem}[section]
\newtheorem{lemma}[theorem]{Lemma}

\theoremstyle{thmstylethree}
\newtheorem{definition}[theorem]{Definition}
\newtheorem{remark}[theorem]{Remark}
\newtheorem{assumption}[theorem]{Assumption}

\crefname{assumption}{Assumption}{Assumptions}
\crefname{hypothesis}{Hypothesis}{Hypotheses}
\crefname{fact}{Fact}{Facts}
\crefname{example}{Example}{Examples}
\numberwithin{equation}{section}

\DeclareMathOperator{\divergence}{div}

\DeclareMathOperator{\diam}{diam}

\makeatletter
\providecommand{\theHALG@line}{\thealgorithm.\arabic{ALG@line}}
\def\@maketitle{%
  \let\footnote\thanks
  \clearemptydoublepage
    \checkoddpage\ifcpoddpage\setlength{\aboveskipchk}{-28.6pt}\else\setlength{\aboveskipchk}{-28.6pt}\fi%for checking oddpage or evenpage top skip%%
  \vspace*{\aboveskipchk}%
  \vspace{\dropfromtop}%
    \vspace*{21pt}
  \if@twocolumn\else\vspace*{-10pt}\fi
    \vspace*{21pt}\par%
  \hbox to \textwidth{%
  \parbox[t]{\titlepagewd}{%
    \ifx\@subtitle\@empty%
    \else%
    {\sffamilyfontcn\fontsize{14}{21}\selectfont\raggedright \@subtitle \par}%
    \vspace{7.5\p@}%
    \fi%
        {\fontsize{12bp}{14pt}\selectfont\bfseries\leftskip0pt plus1fill\rightskip0pt plus1fill {\@title} \par}%%
    \vspace{12\p@}%
        {\fontsize{10bp}{12}\selectfont\scshape\leftskip0pt plus1fill\rightskip0pt plus1fill \@author \par}%
    \vspace{4\p@}
    \vspace{-10\p@}
        {\fontsize{10bp}{12}\selectfont\leftskip0pt plus1fill\rightskip0pt plus1fill\@corresp \par}%
    \vspace{10\p@}
        {\fontsize{10bp}{12pt}\selectfont\leftskip0pt plus1fill\rightskip0pt plus1fill  \par}
    \vspace*{10.5pt}
    \begingroup
    \parindent=0pt
    {\fontsize{9bp}{11}\selectfont\leftskip18pt\rightskip18pt\@abstract\par}
    \endgroup
    }%
  }
  \vspace{13.5\p@}%
  \ifx\@boxedtext\@empty\else\vspace*{12pt}\par\fi%
  \@boxedtext
  \vspace{12\p@ plus 6\p@ minus 6\p@}%
  \vspace{\extraspace}
  \vspace*{-18pt}}
\def\ps@opening{\let\@oddfoot\@empty\let\@evenfoot\@empty
  \let\@oddhead\@empty\let\@evenhead\@empty}
\newcommand{\MSCcodes}[1]{\g@addto@macro\@abstract{\par\vspace*{8pt}%
  {\fontsize{9bp}{11}\selectfont\textit{AMS subject classifications:} #1\par}}}
\makeatother

\begin{document}
\title[A TLFSI Method for Generalized Hermitian Interior Eigenvalues]{A Two-level Filtered Subspace Iterative Method
  for Generalized Hermitian Interior Eigenvalue Problems}
  
% ============================================================
% 作者信息模板（搜索“作者信息模板”即可定位）
% 按姓氏的拉丁字母顺序填写四位作者；顺序不表示贡献大小。
% 第三、第四位作者为通讯作者，以 * 标注。
% 数字上标仅对应单位；同单位作者可共用编号，并删除重复单位行。
% ============================================================
\author{%
  % 作者 1：替换为姓名（名在前，姓在后）
  Qi Li,
  % 作者 2：替换为姓名
  Shi Shu,\\[3pt]
  % 作者 3：通讯作者，保留 *
  Xiaofei Xu\textsuperscript{*}
  and
  % 作者 4：通讯作者，保留 *
  Juan Zhang\textsuperscript{*}
  % 本模板的 \address 放在 \author 内部，才能随作者信息一起显示。
  {\normalfont\address{%
    % 四位作者共同的单位：院系、学校、城市、邮编、国家
    School of Mathematics and Computational Science,
    Hunan Key Laboratory for Computation and Simulation in Science and Engineering, Key Laboratory of Intelligent Computing and Information Processing of Ministry of Education, Xiangtan University, Xiangtan 411105, Hunan, China
  }}%
}
% 通讯邮箱：分别替换邮箱及括号中的姓名，与第三、第四位作者保持一致。
\corresp[*]{Corresponding authors:\\
  \texttt{LittleFFFEI@126.com} (Xiaofei Xu);\\
  \texttt{zhangjuan@xtu.edu.cn} (Juan Zhang).}
% 其他作者邮箱（如投稿要求提供，可填写后按需加入单位信息）：
% 作者 1：author1@example.com
% 作者 2：author2@example.com
% 页眉：替换为四位作者的姓氏，顺序与上面一致。
\authormark{Li, Shu, Xu and Zhang}
% PDF 属性中的作者信息：替换为完整姓名，顺序与上面一致。
\hypersetup{pdfauthor={Qi Li, Shi Shu,
  Xiaofei Xu, Juan Zhang}}
% ==================== 作者信息模板结束 ====================

\abstract{We propose a two-level filtered subspace iterative (TLFSI)
method for computing all eigenpairs whose eigenvalues in the interval of interest $[a,b]$ after finite element discretization of self-adjoint PDE eigenvalue problems.
The method combines a two-level approach with a modern filtered subspace iteration framework.
In this paper, we use contour integral filtering as an example.
The two-level method constructs a coarse space according to the spectral location of the eigenvalues of interest and obtains from it an approximation to the target invariant subspace.
We prove that the distance between the approximate and target subspaces depends only on the separation between the spectral region of interest and the remainder of the spectrum, together with the coarse mesh size \(H\).
Numerical experiments on benchmark elasticity problems confirm the theoretical predictions.}
\MSCcodes{65N25, 65N30, 65N55, 74B05}
\maketitle

% ============================================================
%  1. INTRODUCTION
% ============================================================
\section{Introduction}\label{sec:intro}

Many problems in computational science and engineering require
computing eigenvalues and eigenvectors of partial differential
equation (PDE) eigenvalue problems.
Typical examples include vibration analysis of elastic structures,
electromagnetic cavity resonance, and quantum mechanical systems.
For elliptic differential eigenvalue problems, standard finite element discretizations lead to generalized eigenvalue problems with symmetric positive definite matrix pairs.
% After finite element discretization, such problems reduce to
% finite-dimensional variational eigenvalue problems of the form
% \begin{equation}\label{eq:gep}
%   a(v_h,w_h)=\lambda_h b(v_h,w_h)
%   \qquad\forall w_h\in V_h,
% \end{equation}
% where \(V_h\) is the finite element space and \(b(\cdot,\cdot)\) is the inner product.
In many applications, the eigenvalues of interest lie in a prescribed interior interval of the spectrum, giving rise to what is known as an \emph{interior eigenvalue problem}. Given an interval $[a,b]\subset\mathbb{R}$, the goal is to compute all discrete eigenpairs whose eigenvalues in the interval.

For large-scale eigenvalue problems, classical approaches to interior eigenvalue problems include the
shift-and-invert Lanczos method \cite{PaSa82,BaJePl05},
Davidson-type methods \cite{SlVo96}, SS-RR method \cite{sakurai2003projection}  and filtered subspace iteration (FSI) methods \cite{Po09, li2016thick, ni2026adapolysi}.
Among these, filtered subspace iteration has attracted considerable attention because of its natural parallelism and robustness.
The FSI method constructs a filter function $f(x)$, with $f(x) \approx 1$ for $x \in [a,b]$ and $f(x) \approx 0$ for $x \notin [a,b]$. 
Applying the filter to the finite element operator preserves the eigencomponents associated with \([a,b]\) while suppressing the remaining components, thereby enabling their extraction by subspace iteration. 
It is worth noting that such methods require an admissible initial subspace, and most existing algorithms simply use random initial vectors, which are generally not optimal and do not guarantee that every target eigenvector has a nonzero component in the initial subspace.

On the other hand, two-level methods, or more generally multilevel methods, have also been explored for solving eigenvalue problems.
A classical example is the multigrid method \cite{xie2014multigrid}, which employs a hierarchy of successively refined grids so that the approximations on coarse spaces progressively approach the solution of the original problem.
A more recent example is the two-level preconditioned Helmholtz subspace iterative method proposed in \cite{LiXu23}.
Other related works include \cite{LiXu21,WaXu21,ZhHwCa07}. These methods typically exploit a coarse space to accelerate eigenvalue computations on fine discretizations. However, these methods are typically limited to extremal eigenvalue problems.
However, these methods are typically used only for extreme eigenvalue problems.

% On the other hand, domain decomposition (DD) methods have been successfully applied to eigenvalue problems; see, e.g., \cite{LiXu21,WaXu21,ZhHwCa07} for elliptic eigenvalue problems and \cite{LiXu23} for Maxwell eigenvalue problems.
% In particular, Liang and Xu \cite{LiXu23} proposed a two-level preconditioned Helmholtz subspace iterative (PHSI) method for computing the first several nonzero eigenvalues of Maxwell equations.
% The PHSI method is shown to have a convergence rate that is independent of the mesh size \(h\) and of the internal gaps among the target eigenvalues. However, the method is designed for the smallest nonzero eigenvalues. 
% A more recent two-level additive Schwarz method
% addresses eigenvalues in an interval of interest by combining an additive Schwarz preconditioner with the Jacobi--Davidson framework, but it requires repeated construction of preconditioners during the iteration.

In this paper, we propose a two-level filtered
subspace iterative (TLFSI) method for computing interior eigenvalues of discretized self-adjoint PDE operators. 
% On the one hand, the TLFSI method obtains an improved initial subspace from the coarse space. On the other hand, at each iteration, the filter is applied through complex-shifted variational solves treated with two-level preconditioners.
The TLFSI method obtains an initial subspace from a coarse space that is adaptively constructed according to the spectral location of the eigenvalues of interest and is guaranteed to contain an accurate approximation to the target invariant subspace.
Moreover, we establish convergence with exact filtering and with inexact shifted solves satisfying a residual forcing condition.  
Under filter separation on the discrete spectrum, the subspace contraction factor \(\rho_h\) is determined by the target and complementary filter magnitudes, without requiring a lower bound on the internal gaps among the target eigenvalues.

The remainder of the paper is organized as follows.
The next section introduces the problem of interest and some preliminaries.
The TLFSI algorithm is described in \cref{sec:algorithm}
and its convergence analyzed in \cref{sec:convergence}.
Numerical experiments and conclusions are given in
\cref{sec:applications,sec:conclusions}.

% ============================================================
%  2. MODEL PROBLEMS AND PRELIMINARIES
% ============================================================
\section{Problem and Preliminaries}\label{sec:model}

\subsection{Problem Setting}\label{sec:preliminaries}

We focus on the linear elasticity eigenvalue problem. Let
\(\Omega\subset\mathbb R^3\) be a bounded polyhedral domain, and let
\( n\) denote the outward unit normal. The boundary is decomposed as
\(\partial\Omega=\overline{\Gamma_D}\cup\overline{\Gamma_R}\), where
\(\Gamma_D\) and \(\Gamma_R\) are disjoint relatively open boundary
parts. For a displacement field \( u=(u_1,u_2,u_3)^T\), define
\begin{equation*}%\label{eq:stress-strain}
  \sigma( u)
  =2\mu_L\,\varepsilon( u)
  +\lambda_L\,\divergence( u)\, I,
  \qquad
  \varepsilon( \mathbf{u})
  =\frac12\bigl(\nabla u+(\nabla u)^T\bigr),
\end{equation*}
where
\begin{equation*}%\label{eq:lame}
  \lambda_L=\frac{E\,\upsilon}{(1+\upsilon)(1-2\upsilon)},
  \qquad
  \mu_L=\frac{E}{2(1+\upsilon)},
  \qquad E>0,\quad 0\le\upsilon<\frac12,
\end{equation*}
where \(E\) is the Young's modulus and \(\upsilon\) is the Poisson ratio.
Let \(\uprho\) be a uniformly positive and bounded density. We consider
the elasticity eigenvalue problem
\begin{equation}\label{eq:elasticity-eig}
  \begin{cases}
    -\nabla\cdot\sigma( u)=\lambda\uprho\, u,
      & \text{in }\Omega,\\[4pt]
     u=\mathbf{0},
      & \text{on }\Gamma_D,\\[4pt]
    \sigma( u) n+\eta\, u=\mathbf{0},
      & \text{on }\Gamma_R,
  \end{cases}
\end{equation}
where \(\eta\ge0\).

Let
\[
   \mathbf{H}^1(\Omega):=[H^1(\Omega)]^3,
\qquad
\mathbf{H}_D^1(\Omega)
:=
\{
v\in\mathbf{H}^1(\Omega):
v=\mathbf{0}\text{ on }\Gamma_D
\}.
\]
The trial space \(V\) is determined by the boundary condition. In the
Dirichlet and mixed Dirichlet--Robin cases, \(V= \mathbf{H}_D^1(\Omega)\).
In the pure Robin and pure Neumann cases, \(V= \mathbf{H}^1(\Omega)\). We set
\(\mathcal H= [L^2(\Omega)]^3\) and define
\begin{equation*}%\label{eq:bilinear-a}
  a( u, v)
  :=\int_\Omega\sigma( u):
        \varepsilon( v)\,dx
    +\int_{\Gamma_R}\eta\, u\cdot v\,ds,
  \qquad
  b( u, v)
  :=\int_\Omega\uprho\, u\cdot v\,dx .
\end{equation*}
The weak formulation of \eqref{eq:elasticity-eig} is to find
\((\lambda, u)\in\mathbb R\times V\), with \( u\neq0\), such that
\begin{equation}\label{eq:abstract-pde-eig}
  \begin{cases}
    a( u, v)=\lambda b( u, v) & \forall  v\in V,\\
    \| u\|_b=1,
  \end{cases}
\end{equation}
where \(\| v\|_b=b( v, v)^{1/2}\). When
\(a(\cdot,\cdot)\) is coercive on \(V\), we also write
\(\| v\|_a=a( v, v)^{1/2}\). In the pure Neumann case,
\(a(\cdot,\cdot)\) is positive semidefinite and its kernel is the space
of rigid-body motions. The zero eigenvalues associated with this kernel
are excluded, and the problem is considered on its \(b\)-orthogonal
complement.

% The boundary conditions used in our numerical experiments are summarized
% as follows.
% \begin{itemize}
% \item[(i)] \textbf{Pure Dirichlet boundary condition.}
% Let \(\Gamma_D=\partial\Omega\), \(\Gamma_R=\emptyset\), and
% \(V= H_D^1(\Omega)\). Then \(a(\cdot,\cdot)\) is coercive on \(V\),
% and the corresponding finite element stiffness matrix is symmetric
% positive definite.

% \item[(ii)] \textbf{Pure Robin boundary condition.}
% Let \(\Gamma_D=\emptyset\), \(\Gamma_R=\partial\Omega\), \(\eta>0\), and
% \(V= H^1(\Omega)\). The Robin term removes the rigid-body kernel, so
% \(a(\cdot,\cdot)\) is coercive on \(V\). The corresponding finite element
% stiffness matrix is symmetric positive definite.

% \item[(iii)] \textbf{Pure Neumann boundary condition.}
% Let \(\Gamma_D=\emptyset\), \(\Gamma_R=\partial\Omega\), \(\eta=0\), and
% \(V= H^1(\Omega)\). Then \(a(\cdot,\cdot)\) is positive semidefinite
% on \(V\), with rigid-body motions forming its kernel. The interval of
% interest is therefore chosen away from the zero eigenvalues.

% \item[(iv)] \textbf{Mixed Dirichlet--Robin boundary condition.}
% Let both \(\Gamma_D\) and \(\Gamma_R\) be nonempty, with \(\eta>0\) on
% \(\Gamma_R\), and let \(V= H_D^1(\Omega)\). Then \(a(\cdot,\cdot)\) is
% coercive on \(V\), and the corresponding finite element stiffness matrix
% is symmetric positive definite.
% \end{itemize}

For the coercive cases, define the solution operator
\(T:\mathcal H\to V\) by
\[
  a(Tw, v)=b( w, v)
  \qquad \forall  v\in V,\quad  w\in\mathcal H .
\]
By the compact embedding \(V\hookrightarrow\mathcal H\), the operator
\(T\), viewed as an operator on \(\mathcal H\), is compact. Let
\(\{(\lambda_i, u_i)\}_{i=1}^{\infty}\) denote the eigenpairs of
\eqref{eq:abstract-pde-eig}, ordered such that
\(\lambda_1 \le \lambda_2 \le \cdots\). Then
\[
  T u_i=\lambda_i^{-1} u_i .
\]
Given an interval of interest \([a,b]\subset\mathbb R\), we define the
target invariant subspace by
\[
  \mathcal E
  :=\operatorname{span}\{ u_i:\lambda_i\in[a,b]\}.
\]
We decompose \(V=\mathcal E\oplus\mathcal E^\perp\), where the
orthogonal complement is taken with respect to the inner product used in
the analysis. In the pure Neumann case, the same definitions are applied
after restricting the problem to the \(b\)-orthogonal complement of the
rigid-body kernel.

\subsection{Finite element discretization and eigenvalue problems}
\label{sec:fem}

Let $\{\tau_h\}_{h>0}$ be a family of conforming, boundary-fitted simplicial triangulations of $\Omega$. We assume that the family is shape regular and quasi-uniform, and set
$h:=\max_{K\in\tau_h}\diam(K)$. Let $V_h\subset V$ be a conforming finite element space and set $n_h:=\dim V_h$. 
% For Maxwell problems,
% $V_h$ is a curl-conforming N\'{e}d\'{e}lec edge-element space; 
For elasticity, it is a vector-valued $H^1$-conforming Lagrange space.
The corresponding essential boundary degrees of freedom are set to zero. 
The Galerkin finite element discretization of
\eqref{eq:abstract-pde-eig} on \(V_h\) is to find
$(\lambda^h, u^h)\in\mathbb R\times V_h$, with $ u^h\neq0$, such that
\begin{equation}\label{eq:disc-eig}
  a( u^h, v^h)=\lambda^h b( u^h, v^h)
  \qquad\forall  v^h\in V_h.
\end{equation}
% Exact integration is assumed in the analysis.
Let
\(
  \{(\lambda_j^h, u_j^h)\}_{j=1}^{n_h}
\)
be its eigenpairs, including multiplicities, such that
\[
  a( u_i^h, u_j^h)=\lambda_i^h b( u_i^h, u_j^h)
  =\lambda_i^h\delta_{ij},
\]
where the eigenvalues are ordered as
\(\lambda_1^h\le \lambda_2^h\le\cdots\le \lambda_{n_h}^h\).
Define \(T_h:\mathcal H\to V_h\) by
\[
a(T_h  w, v)=b( w, v)\qquad\forall  v\in V_h,\  w\in\mathcal H.
\]
Due to $T_h$ is symmetric and compact, we have $T_h  u_i^h = (\lambda_i^h)^{-1}  u_i^h,\ i = 1,\cdots, n_h$.
Define \(A_h:V_h\to V_h\) by
\begin{equation*}%\label{eq:disc-operator}
  b(A_h  v^h, w^h)=a( v^h, w^h)
  \qquad \forall  v^h, w^h\in V_h .
\end{equation*}
Decompose $V_h=\mathcal E_h\oplus\mathcal E_h^{\perp},$ where
\(
  \mathcal E_h:=\operatorname{span}\{  u_i^h:i\in\mathcal I_h\}
\)
and $\mathcal E_h^{\perp}$ is the $b$-orthogonal complement $\mathcal E_h$, where \(\mathcal I_h:=\{j:\lambda_j\in [a, b]\}.\)
The \(b\)-orthogonal projector \(P_h:V_h\to\mathcal E_h\) is defined by
\[
  b(P_h v^h, w^h)=b( v^h, w^h)
  \qquad \forall  w^h\in\mathcal E_h,
\]
and \(P_h^\perp:=I-P_h\) is the \(b\)-orthogonal projector onto
\(\mathcal E_h^{\perp}\). 
% Let the filter function \(f:\mathbb{R}\to\mathbb{R}\) be a scalar filter satisfying
% \(
% f(x)>\frac{1}{2},\ x\in[a,b],
% \)
% and
% \(
% |f(x)|<\frac{1}{2},\ x\in\mathbb{R}\setminus[a,b].
% \)
% Define the discrete filter operator \(F_h:V_h\to V_h\) by
% \(
% F_h=f(A_h)=f(T_h^{-1}).
% \)
% It follows that
% \[
% F_hu_i^h=\gamma_i^h u_i^h,
% \]
% where
% \(
% \gamma_i^h=f(\lambda_i^h),
% \ 
% \gamma_1^h\geq\gamma_2^h\geq\cdots\geq\gamma_{n_h}^h.
% \)
% Similarly, define the continuous filter operator \(F:V\to V\) by
% \(
% F=f(T^{-1}).
% \)
% Then
% \[
% Fu_i=\gamma_i u_i,
% \]
% where
% \(
% \gamma_i=f(\lambda_i),
% \
% \gamma_1\geq\gamma_2\geq\cdots,
% \)
% are the eigenvalues of \(F\).

To quantify the distance between two subspaces, we use principal angles.
\begin{definition}{\rm\cite{GoVa13}}
\label{def:subspace-angle}
Let \(\mathcal S\) and \(\mathcal T\) be finite-dimensional subspaces of an inner-product space with the same dimension. Denote by \(P_{\mathcal S}\) and \(P_{\mathcal T}\) the corresponding \(b\)-orthogonal projection operators onto \(\mathcal S\) and \(\mathcal T\), respectively. We denote by
\[
  \theta_b(\mathcal S,\mathcal T):=\sin\Theta(\mathcal S,\mathcal T) = \|P_{\mathcal S} - P_{\mathcal T}\|_b.
\] 
Similarly, \(\theta_a(\mathcal S,\mathcal T)\) can be defined.
\end{definition}

By Theorem 2.5.1 in \cite{golub2013matrix}, we can compute \(\theta_b(\mathcal S,\mathcal T)\) as
\(
\theta_b(\mathcal S,\mathcal T)
=
\|P_{\mathcal S}(I-P_{\mathcal T})\|_b
=
\|(I-P_{\mathcal S})P_{\mathcal T}\|_b .
\)
For certain elliptic eigenvalue problems, approximation properties between the continuous and discrete problems have been established in some works \cite{liang2024,LiXu26}. 
Following the same framework, we present the corresponding approximation properties for the continuous problem \eqref{eq:abstract-pde-eig} and the discrete problem
\eqref{eq:disc-eig}.

\begin{theorem}\label{theo:T_and_Th}
    Let \(\Omega\) be a bounded convex polygonal domain. Then, as \(h\to0\), the eigenvalues \(\lambda_{i}^h,\ldots,\lambda_{i+s}^h\) of the discrete eigenvalue problem \eqref{eq:disc-eig} converge to the corresponding eigenvalues \(\lambda_{i},\ldots,\lambda_{i+s}\) of the continuous eigenvalue problem \eqref{eq:abstract-pde-eig}. Moreover, there exists \(h_0>0\) such that, for any \(0<h<h_0\),
    \begin{equation}\label{lambda_j}
        0\le \lambda_j^h - \lambda_j \le Ch^2,
    \end{equation}
    and
    \begin{equation}\label{theta_b}
       \theta_b(\mathcal{E}, \mathcal{E}_h) \le C h^2,
    \end{equation}
    where \(C\) depends on
    \(
    d=\min\{d^+,d^-\},
    \)
    where \(d^+ = \lambda_{i+s+1} - \lambda_{i+s}\) and \(d^- = \lambda_{i} - \lambda_{i-1}\).    
\end{theorem}

\begin{proof}
Let \(\mu_j=\lambda_j^{-1}\) and
\(\mu_j^h=(\lambda_j^h)^{-1}\).
By the conforming Galerkin min--max principle,
\(\lambda_j\le\lambda_j^h\). Moreover, the min--max principle
for the \(b\)-self-adjoint solution operators \(T\) and \(T_h\),
together with the estimate
\(\|T-T_h\|_b\le Ch^2\) from \cite{lee2018}, yields
\[
    |\mu_j-\mu_j^h|\le \|T-T_h\|_b\le Ch^2.
\]
Hence,
\[
    0\le \lambda_j^h-\lambda_j
    =\frac{\mu_j-\mu_j^h}{\mu_j\mu_j^h}
    \le C (\lambda_{i+s}^h)^2 h^2
    \le C h^2,
    \qquad j=i,\ldots,i+s,
\]
which proves \eqref{lambda_j}.

Let \(\Gamma\) be a positively oriented circle enclosing
\(\mu_i,\ldots,\mu_{i+s}\) and no other points of the spectrum
of \(T\). For sufficiently small \(h\), \(\Gamma\) also encloses
precisely \(\mu_i^h,\ldots,\mu_{i+s}^h\) among the spectral
points of \(T_h\). Define the spectral projections
\[
    Z:=\frac{1}{2\pi\mathrm{i}}
       \int_{\Gamma}(zI-T)^{-1}\,dz,
    \qquad
    Z_h:=\frac{1}{2\pi\mathrm{i}}
       \int_{\Gamma}(zI-T_h)^{-1}\,dz.
\]
Then \(Z\) and \(Z_h\) project onto
\(\mathcal E=\operatorname{span}\{u_i,\ldots,u_{i+s}\}\)
and
\(\mathcal E_h=\operatorname{span}\{u_i^h,\ldots,u_{i+s}^h\}\),
respectively. Combining \(\|T-T_h\|_b\le Ch^2\) with the
argument in the appendix of \cite{liang2024} gives
\eqref{theta_b}.
\end{proof}

\subsection{Contour-integral spectral filter}
\label{sec:contour}
The filter function \(f\) may be constructed as either a polynomial \cite{li2016thick, xu2026adaptive} or a rational \cite{Po09, wang2026shifted} function, with the latter represented as the quotient of two polynomials. Contour-integral filters are among the most widely used rational filters because they provide a high degree of parallelism and extend naturally to generalized eigenvalue problems. This paper focuses on the contour-integral filter function.

For the interval $[a,b]$, we define the contour integral function as
\begin{equation*}%\label{eq:rho-scalar}
  \chi(x)=-\frac{1}{2\pi i}\oint_{\mathcal C}\frac{1}{x-z}\,dz
  =
  \begin{cases}
    1, & x \in \Gamma,\\[4pt]
    0, & x \notin \Gamma,
  \end{cases}
\end{equation*}
where $\mathcal{C}$ encloses the eigenvalues in $[a,b]$ and no other discrete eigenvalues, and \(\Gamma\) denotes the region enclosed by \(\mathcal{C}\). 
The corresponding spectral projector is
\begin{equation*}%\label{eq:contour}
  \chi(A_h)
  =
  -\frac{1}{2\pi i}\oint_{\mathcal C}
  (A_h-zI)^{-1}\,dz .
\end{equation*}
The spectral projector onto \(\mathcal E_h\) satisfies
\[
  \chi(A_h)v_h
  =\sum_{i\in\mathcal I_h} b(v_h,u_i^h)u_i^h,
  \qquad v_h\in V_h.
\]
If the initial subspace \(U^0\) satisfies
\(\dim\chi(A_h)U^0=\dim\mathcal E_h\), then
\(\chi(A_h)U^0=\mathcal E_h\), so exact projection recovers the
target subspace in one step. In practice, the projector is
approximated by discretizing its contour integral representation.
% {\color{blue} For any \(v_h\in V_h\),
% \[
% A_h v_h
% =
% \sum_{i=1}^{n_h}\lambda_i^h\, b(v_h,u_i^h) u_i^h .
% \]
% The spectral projector associated with
% \(\mathcal E_h\) is given by
% \[
% \rho(A_h)v_h
% =
% \sum_{i\in\mathcal I_h} b(v_h,u_i^h) u_i^h ,
% \qquad v_h\in V_h .
% \]
% Therefore, for any initial vector \(v_h^0\in V_h\) whose component in
% \(\mathcal E_h\) is nonzero,
% \(\rho(A_h)v_h^0\in\mathcal E_h\). More generally, for any initial subspace represented by \(U^0=\{u_1^0,\ldots,u_s^0\}\), choose an
% arbitrary enumeration \(\mathcal I_h=\{i_1,\ldots,i_s\}\).  If the
% matrix
% \[
% \bigl(b(u_j^0,u_{i_k}^h)\bigr)_{k,j=1}^s
% \]
% is nonsingular, we have
% \[
% \rho(A_h)U^0 \subset \mathcal E_h,
% \qquad \dim \rho(A_h)U^0=s .
% \]
% Thus, in this ideal case, the subspace iteration converges in one step.
% In practice, however, the action of \(\rho(A_h)\) cannot be computed
% inexpensively and must be approximated. One standard approach is to
% approximate the spectral projector by a discretized contour integral.
% }

The contour integral is approximated by a quadrature rule with \(n_q\)
nodes \(\{z_j\}_{j=1}^{n_q}\) and weights
\(\{w_j\}_{j=1}^{n_q}\).  Define the rational
function
\(
  f_{n_q}(x):=
  \sum_{j=1}^{n_q}\frac{w_j}{x-z_j}
\)
and the filter matrix \(F_h\) by
\begin{equation}\label{eq:quad}
  F_h^{(n_q)}
  :=
  f_{n_q}(A_h)
  =
  \sum_{j=1}^{n_q}
  w_j
  (A_h-z_j I)^{-1} .
\end{equation}

For a fixed quadrature rule, we write \(F_h=F_h^{(n_q)}\) and
\(f=f_{n_q}\) when no ambiguity is possible.
As \(n_q\to\infty\), the quadrature-based rational filter \(f_{n_q}\)
converges pointwise to \(\rho\) on any fixed discrete spectrum separated
from \(\mathcal C\) \cite{Bornemann11}.

We use a contour symmetric about the real axis. A convenient
parameterization is the ellipse
\begin{equation*}%\label{eq:ellipse-contour}
  z(\theta)=c+\alpha\cos\theta+i\beta\sin\theta,
  \qquad
  c=\frac{a+b}{2},\quad
  \alpha=\frac{1+\delta_c}{2}(b-a),\quad
  \beta=\eta\alpha,
\end{equation*}
where $\delta_c>0$ and $\eta\in(0,1]$. The contour must enclose
$[a,b]$ and exclude unwanted eigenvalues. We discretize the upper half
of the contour by a standard quadrature rule, such as Gauss--Legendre or
trapezoidal quadrature \cite{Po09,TaPo14}, and obtain the remaining
nodes and weights by conjugation.
%Zolotarev rules are another alternative \cite{guttel2015zolotarev}. 
For a convergent quadrature family on a fixed contour, the scalar
approximation converges as \(n_q\to\infty\).  Changing \(\beta\) changes both the
filter separation and the conditioning of the shifted systems. The
contour and quadrature parameters used in each experiment are therefore
reported together with the shifted-solver settings.

For a real \(q^h\in V_h\), conjugate pairing makes the filtered
function real.
Number the quadrature nodes so that
$z_j = \bar{z}_{j+n_q/2}$ and
$w_j = \bar{w}_{j+n_q/2}$ for $j = 1,\ldots,n_q/2$, we have
\begin{equation}\label{eq:real-filter}
  F_hq^h
  =
  2\,\operatorname{Re}\Bigl(
  \sum_{j=1}^{n_q/2}
  w_j(A_h-z_jI)^{-1}q^h
  \Bigr).
\end{equation}
Thus one filter application requires \(n_q/2\) complex shifted
variational solves for each real input function.

% ============================================================
%  3. THE TWO-LEVEL TLFSI ALGORITHM
% ============================================================
\section{The TLFSI algorithm}\label{sec:algorithm}

The convergence of filtered subspace iteration relies on the ability of the filter to amplify the components of the approximate subspace associated with the target invariant subspace. For any vector
\[
x=\sum_{j=1}^{n_h}\xi_j  u_j^h\in V_h,
\]
we have the decomposition
\[
\begin{aligned}
y=F_hx
&=\sum_{j=1}^{n_h}\xi_jF_h u_j^h \\
&= \sum_{j \in \mathcal I_h} \xi_j\gamma_j^h  u_j^h + \sum_{j \notin \mathcal I_h} \xi_j\gamma_j^h  u_j^h \\
&= y_1 + y_2,
\end{aligned}
\]
where \(y_1\in\mathcal E_h\) and \(y_2\in\mathcal E_h^\perp\). The component of \(y\) in \(\mathcal E_h\) is determined jointly by the corresponding component of \(x\) and the eigenvalues of \(F_h\). The latter can be controlled through the accuracy of the filter, whereas the former depends on the choice of the initial subspace.

\subsection{Construction of initial subspace}
\label{sec:init-subspace}

Let \(\tau_H=\{\Omega_l\}_{l=1}^N\) be a conforming, shape-regular, and quasi-uniform coarse triangulation of \(\Omega\), and set
\[
  H_l:=\diam(\Omega_l),\qquad
  H:=\max_{1\le l\le N}H_l .
\]
Assume that \(H>h\) and that \(\tau_h\) is obtained by conforming refinement of \(\tau_H\). 
Using compatible finite element families gives
the nested spaces \(V_H\subset V_h\subset V\). Set \(n_H:=\dim V_H\).
The coarse Galerkin problem is to find
\((\lambda^H, u^H)\in\mathbb R\times V_H\), with \( u^H\neq0\), such that
\begin{equation*}%\label{eq:coarse-eig}
  a( u^H, v^H)=\lambda^H b( u^H, v^H)
  \qquad\forall  v^H\in V_H .
\end{equation*}
Define \(T_H: \mathcal{H} \to V_H\) by
\[
  a(T_H w, v)=b( w, v)
  \qquad \forall  v\in V_H,\  w \in \mathcal{H},
\]
and define \(A_H:V_H\to V_H\) by
\[
  b(A_H v^H, w^H)=a( v^H, w^H)
  \qquad \forall  v^H, w^H\in V_H .
\]

Let \(\{(\lambda_j^H, u_j^H)\}_{j=1}^{n_H}\) be the coarse eigenpairs ordered so that
\(\lambda_1^H \le \lambda_2^H \le \cdots \le \lambda_{n_H}^H\).
% Define \(F_H=f(A_H)\), and let the eigenvalues of \(F_H\) be denoted by
% \(\gamma_i^H=f(\lambda_i^H)\), which are ordered as
% \(
% \gamma_1^H\ge\gamma_2^H\ge\cdots\ge\gamma_{n_H}^H .
% \)
% As discussed in Section~\ref{sec:contour}, this ordering places the eigenvalues associated with the interval of interest at the beginning of the ordered sequence. 
Suppose that \(A_h\) has \(e\) eigenvalues in the interval \([a,b]\).
Let \(i\) be the smallest index such that \(\lambda_i^h\in[a,b]\).
By Theorem \ref{theo:T_and_Th}, we have \( u_i^H,\ldots, u_{i+e-1}^H\) can be regarded as approximations to \( u_i^h,\ldots, u_{i+e-1}^h\), respectively. To further improve these approximations in terms of the residual, we perform one Rayleigh--Ritz procedure for \(A_h\) over the subspace
\(
\operatorname{span}\{ u_i^H,\cdots, u_{i+e-1}^H\}.
\)
% In the following section, we provide a theoretical analysis to justify the effectiveness of this choice.

In practice, the exact number \(e\) of eigenvalues in the interval of interest is not available at negligible cost. 
Nevertheless, several inexpensive estimation or eigenvalue-counting techniques can provide a reliable upper bound \cite{di2016efficient, LiSaYa16, xi2018fast}. 
These estimation techniques yield an estimate \(\tilde e \approx e\). We then choose \(\zeta\) such that \(s=\lceil \zeta \tilde e\rceil \ge e,\) and use \(s\) as the dimension of the search subspace.
This choice has two advantages. First, an enlarged subspace improves the robustness and convergence of FSI. 
Second, the starting index of the target eigenvectors is not known a priori, and the coarse eigenvectors should be selected so that, whenever possible, their indices cover those of the target eigenvectors. 

We consider two strategies for selecting the coarse subspace used to approximate the target eigenspace. 
The first strategy is the most robust, as it ensures that the indices of the selected coarse eigenpairs cover those of the target eigenpairs, although at a relatively higher computational cost. 
We first use an inertia count to determine the number of eigenvalues below the left
endpoint \(a\), thereby identifying the index \(i\) of the first
eigenvalue greater than or equal to \(a\). 
We then use an existing eigensolver to compute the \(i\)-th through the \((i+s-1)\)-st
eigenvectors of the coarse eigenvalue problem and take
\[
  \mathcal V_0^{\rm IDX}
  =
  \operatorname{span}\{u_i^H,\ldots,u_{i+s-1}^H\}
\]
as the initial subspace. 
The second strategy avoids identifying the exact target indices and thereby substantially reduces the additional initialization cost. However, it cannot guarantee that the selected coarse eigenpairs cover all indices corresponding to the target eigenpairs.
Since, for sufficiently small mesh sizes, the conforming finite element eigenvalues satisfy \(\lambda_j^H\ge \lambda_j^h\) for all \(j\), the coarse eigenvalues corresponding to the target eigenvalues are shifted toward the high-frequency part of the spectrum. We therefore take \(\ell\) as the first
index such that \(\lambda_\ell^H\ge a\), and use
\[
  \mathcal V_0^{\rm INT}
  =
  \operatorname{span}\{ u_\ell^H,\cdots, u_{\ell+s-1}^H\}
\]
as the initial subspace. 
We denote the complete algorithms based on the first and second strategies by TLFSI-IDX and TLFSI-INT, respectively.
Figure \ref{fig:init-subspace} provides a schematic illustration of the above idea. 

\begin{figure}[!htbp]
    \centering
    \includegraphics[width=0.8\linewidth]{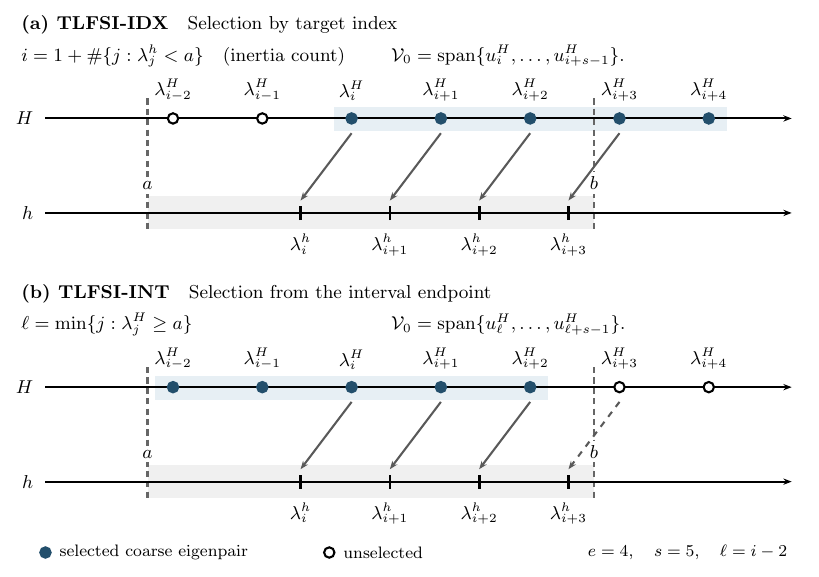}
    \caption{Selection of a coarse-level subspace for approximating the fine-level invariant subspace associated with the interval of interest $[a,b]$.}
    \label{fig:init-subspace}
\end{figure}

\begin{figure}[!htbp]
    \centering
    \includegraphics[width=0.7\linewidth]{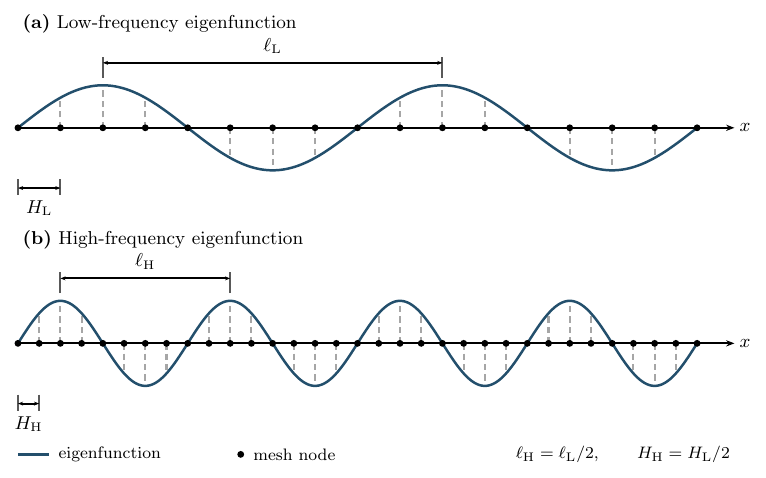}
    \caption{\rm Choice of the coarse mesh size \(H\) relative to \(\ell\) of the eigenfunctions of interest: (a) low-frequency case; (b) high-frequency case.}
    \label{fig:eigen_function}
\end{figure}

\begin{remark}
    The additional cost of TLFSI-IDX mainly arises from the \(LDL^T\) factorization of the original problem, which is performed using MUMPS. Although the matrix to be factorized is real and sparse, the computational cost may still reach \(\mathcal{O}(n_h^{1.5})\). In contrast, the cost of solving the projected subproblem is negligible. A less expensive alternative is to estimate the index \(\ell\) from the eigenvalue problem on the \(2h\) grid. This strategy can substantially reduce the computational cost while retaining robustness, but it remains less efficient than TLFSI-INT. We therefore mention this alternative only briefly.
\end{remark}

\subsection{Selection of $H$}\label{sec:select_H}

For the linear-elastic eigenvalue problem, the eigenvalue \(\lambda\) is related to the natural angular frequency \(\omega\) by \(\lambda=\omega^2\). Hence, the interval of interest \([a,b]\) corresponds to the frequency band \([\sqrt a,\sqrt b]\), and the shortest spatial scale of interest occurs at \(\omega_{\max}=\sqrt b\). In a homogeneous isotropic elastic medium, the shear- and pressure-wave velocities are
\[
c_S=\sqrt{\frac{\mu}{\uprho}},
\qquad
c_P=\sqrt{\frac{\lambda_{\mathrm L}+2\mu}{\uprho}},
\]respectively. Since \(c_S<c_P\), the shear wave has the shorter wavelength at a fixed frequency and therefore determines the spatial resolution required by the coarse mesh \cite{achenbach2012wave}.
Although the eigenfunctions on a bounded domain are standing-wave modes shaped by the geometry and boundary conditions, their local oscillatory scales are governed by the elastic principal symbol. The shortest wavelength associated with the target interval is consequently estimated by
\(
\ell_{\min}
=
\frac{2\pi}{\kappa_S(b)},
\)
where
\(
\kappa_S(b)
=
\frac{\sqrt b}{c_S}
=
\sqrt{\frac{b\uprho}{\mu}}.
\)
Motivated by the elements-per-wavelength criteria used in finite element simulations of elastic waves \cite{geevers2018dispersion}, we select the coarsest nested mesh satisfying
\begin{equation}\label{equ:compute_H}
H\kappa_S(b)\le \tau.    
\end{equation}
Equivalently, the shortest shear wavelength is represented by at least \(2\pi/\tau\) coarse elements. The parameter \(\tau\) is chosen numerically in Section~\ref{sec:applications}. Figure \ref{fig:eigen_function} illustrates that, based on the wavelength \(\ell\) of the eigenfunctions of interest, an appropriate coarse mesh size \(H\) is selected according to \eqref{equ:compute_H}.

\subsection{Implementation}
\label{sec:TLFSI-iteration}

The pseudocode of the algorithm is presented in Algorithm~\ref{alg:solve}. 
Lines~\ref{estimate_eigsnum} and \ref{init} are carried out as described in Sections~\ref{sec:init-subspace} and \ref{sec:select_H}, respectively. 
% However, an excessively small \(H\) may incur substantial overhead in the setup stage. 
% We therefore prevent \(H\) from becoming too small by introducing a threshold \(\text{tol}_H\). 
The setup cost mainly consists of two parts: the construction of the coarse space and the solution of the eigenvalue problem on the coarse space. 
The latter is negligible compared with the former and is computed using the ARPACK \cite{lehoucq1998arpack} routines \texttt{dsaupd} and \texttt{dseupd}.

% In practice, computing the exact number of eigenvalues in the interval of interest is computationally expensive for large-scale eigenvalue problems. Therefore, various algorithms have been developed to provide an estimate of this number \cite{di2016efficient, LiSaYa16, xi2018fast}. Since these algorithms are not the focus of this work, we assume that an estimate \(s\ge e\) is available and use it as the dimension of the initial subspace.
% It is worth noting that, although a larger initial subspace is not required, it can enhance the convergence behavior of the algorithm.
% We retain the \(s\) coarse eigenfunctions whose eigenvalues are closest to the interval of interest, and perform a \(b\)-orthonormal Rayleigh--Ritz extraction on the fine grid. The details of this initialization procedure are provided in Section~\ref{sec:init-subspace}.
In the main iteration loop, the constructed filter is applied to the approximate subspace, followed by an orthogonalization procedure to improve the numerical stability of the algorithm. A Rayleigh--Ritz procedure is then performed on the filtered subspace to accelerate convergence.
Ritz pairs are selected by their Ritz values in \([a,b]\), and convergence is tested by residual functionals of the selected pairs.

We further incorporate implicit deflation and spurious eigenvalue detection techniques into the algorithm. The implicit deflation strategy locks the converged eigenvectors and removes them from the approximate subspace. In each subsequent iteration, the updated approximate subspace is orthogonalized against the locked eigenvectors, see Lines \ref{filtering} and \ref{deflation}.

The spurious eigenvalue detection strategy is based on a relaxed residual threshold \(\text{tol}_s\). 
After some eigenpairs have converged, spurious eigenvalues are detected using \(\text{tol}_s\). Eigenpairs with residuals above \(\text{tol}_s\) are identified as spurious and excluded from the convergence check.

\begin{algorithm}[htbp!]
\caption{The TLFSI method for interior finite element eigenvalues}
\label{alg:solve}
\begin{algorithmic}[1]
\Require Interval \([a,b]\); quadrature data 
\(\{z_j,w_j\}_{j=1}^{n_q/2}\); tolerances $\text{tol}$.
\Ensure Approximate eigenpairs $(\Lambda, U)$.
\State\label{estimate_eigsnum}
% {\color{blue}
% Estimate the number $\ilde e$ of target eigenvalues in \([a,b]\).
% If $\tilde e=0$, return.
% Set the iteration subspace dimension
% $
% s=\left\lceil 1.8\,\widetilde e\right\rceil .
% $
% }
Estimate an upper bound \(s\) for \(e\), the number of target eigenvalues in \([a,b]\). If \(s=0\), return.
\State\label{init}
\(U^0 \gets \mathrm{Initial}([a,b], h, s)\)
(Algorithm~\ref{alg:init}).
\State
Set \(k = 0,\ t = 0\), \(\Lambda = \varnothing\), \( U = \{ 0\}\)
and \(P_\perp=I\).
\Repeat
  \State\label{filtering}
  Solve the eigenvalue problem on \(W^{k+1}\):
  \begin{equation}\label{eq:RR}
    a(u_i^{k+1},v)
    =
    \lambda_i^{k+1}b(u_i^{k+1},v),
    \qquad
    v\in W^{k+1},
    \qquad
    b(u_i^{k+1},u_j^{k+1})=\delta_{ij},
  \end{equation}
  \qquad where
  \[
  W^{k+1}
  =
  \operatorname{span}
  \left\{
      P_\perp F_h U^k
  \right\}.
  \]
  % Set
  % \[
  % U^{k+1}
  % =
  % \operatorname{span}
  % \{u_i^{k+1}\}_{i=1}^{s-t}.
  % \]
  \State\label{deflation}
  \((U^{k+1},U,\Lambda,P_\perp,t,\mathcal I_{k+1})\)
  \(\gets\)
  \(\mathrm{Deflation}
  \big(
  \{\lambda_i^{k+1}\}_{i=1}^{s-t},
  \{u_i^{k+1}\}_{i=1}^{s-t},
  t,\text{tol}
  \big)\)
  (Algorithm~\ref{alg:deflation}).
  \State \(k\gets k+1\).
\Until{
\(\mathcal I_k=\varnothing\)
}
\end{algorithmic}
\end{algorithm}

\begin{algorithm}[htbp!]
\caption{Generate the initial subspace \(U^0\)}
\label{alg:init}
\begin{algorithmic}[1]
\Require Interval \([a,b]\); fine mesh size \(h\); the target count $s$; tolerance $\text{tol}_H$.
\Ensure The initial subspace \(U^0\).
\If{TLFSI-IDX}
  \State \(\ell = \min\{j:\lambda_j^h\ge a\}.\)
\ElsIf{TLFSI-INT}
  \State \(\ell = \min\{j:\lambda_j^H\ge a\}.\)
\EndIf
\State\label{comput_H} Choose the largest \(H\) satisfying \eqref{equ:compute_H}, the mesh-nesting condition and $n_H \ge \ell + s - 1$.
\If{$H > h$}
\State\label{Solve_coarse_prob} Solve
  \[
  a( u_i^H, v)=\lambda_i^Hb( u_i^H, v),\ i\in \mathcal{I}_s,\  v \in V_H,\ \text{and}\ b( u_i^H, u_j^H) =  \delta_{ij},
  \]
  where \(\mathcal I_s = \{\ell,\ell+1,\cdots,\ell+s-1\}\).
  Set \(\mathcal E_H = \text{span}\{  u_j^H \}_{j \in \mathcal I_s}.\)
\State\label{init-coarse} Solve the eigenvalue problem on \(\mathcal E_H\):
\[
    a( u_i^0, v)=\lambda_i^0b( u_i^0, v),\  v \in \mathcal E_H,\ \text{and}\ b( u_i^0, u_j^0) = \delta_{ij}.
\]
Set \( U^0 = \text{span}\{  u_i^0 \}_{i=1}^s. \)
\Else
\State Generate
\(
U^0=\operatorname{span}\{ u_i^0\}_{i=1}^s,
\
 u_i^0 \overset{\mathrm{i.i.d.}}{\sim}\mathcal{N}(0,I).
\)
\EndIf
\end{algorithmic}
\end{algorithm}

\begin{algorithm}[htbp!]
\caption{Deflation of the converged eigenpairs}
\label{alg:deflation}
\begin{algorithmic}[1]
\Require
The eigenpairs
\(\{\lambda_i^{k+1}\}_{i=1}^{s-t}\) and
\(\{u_i^{k+1}\}_{i=1}^{s-t}\);
the locking count \(t\);
tolerance $\text{tol}$.
\Ensure
The updated \(U^{k+1}\), \(U\), \(\Lambda\),
\(P_\perp\), \(t\), and
\(\mathcal I_{k+1}\).
\State
Set
\[
\mathcal J_{k+1}
=
\{i:\lambda_i^{k+1}\in[a,b]\}.
\]
\State
Set \(\mathcal I_{k+1}=\varnothing\).
\For{$i\in\mathcal J_{k+1}$}
\State
Compute the residual
\[
r_i^{k+1}
=
A_hu_i^{k+1}
-
\lambda_i^{k+1}u_i^{k+1}.
\]
\If{
\(\|r_i^{k+1}\|_2<\text{tol}\)
}
\State
Move \(u_i^{k+1}\) from \(U^{k+1}\) to \(U\),
add \(\lambda_i^{k+1}\) to \(\Lambda\),
and set \(t\gets t+1\).
\Else
\State
Add \(i\) to the active index set
\(\mathcal I_{k+1}\).
\EndIf
\EndFor
\State
% {\color{blue}
% Let the locked vectors be
% \[
% U
% =
% \operatorname{span}\{u_1,\ldots,u_t\},
% \qquad
% b(u_i,u_j)=\delta_{ij}.
% \]
Update 
\[
P_\perp v
=
v-\sum_{j=1}^{t} b(v,u_j)u_j, \quad u_j \in U.
\]
If \(t=0\), set \(P_\perp=I\).
% }
\end{algorithmic}
\end{algorithm}

% ============================================================
%  4. CONVERGENCE ANALYSIS
% ============================================================
\section{Convergence analysis}\label{sec:convergence}

Filtered subspace iteration with Rayleigh--Ritz extraction is a standard
framework for rational eigenvalue computations; see, for example,
\cite{TaPo14,GuPoTaVi15,AuTr15,xi2016computing} and the projection
results in \cite[Chapters~4 and~6]{saad2011numerical}.  The estimates
below are stated on the finite element space \(V_h\).

% \subsection{Notation and filter separation}

In this section, we suppose that \(s=e=\dim\mathcal E_h\) and that the target coarse and fine eigenpairs have the same index set.  
% In practice, \(s\) may be an
% upper estimate for the target count. 
The rational filter \(F_h\) is defined by \eqref{eq:quad}. Note that \(\mathcal E_h\) and \(\mathcal E_h^{\perp}\) are invariant under \(F_h\).  
Let
\[
  F_{h,\mathcal E}:=F_h|_{\mathcal E_h},\qquad
  F_{h,\perp}:=F_h|_{\mathcal E_h^{\perp}} .
\]
Define the target and complementary filter magnitudes by
\[
  \alpha_h:=\min_{j\in\mathcal I_h}|f(\lambda_j^h)|,
  \qquad
  \beta_h:=\max_{j\notin\mathcal I_h}|f(\lambda_j^h)|.
\]

\begin{assumption}\label{ass:conv-filter}
For the given finite element space \(V_h\), the filter satisfies
\[
  0\le\beta_h<\alpha_h,
  \qquad
  \rho_h:=\frac{\beta_h}{\alpha_h}<1 .
\]
\end{assumption}

The rational filter is chosen so that its scalar values on the target
eigenvalues dominate those on the complementary spectrum.  For
contour-integral filters this is the usual separation produced by an
off-real contour enclosing the interval; the circular-contour estimates
in \cite[Theorems~3.1--3.2]{TaPo14} give one representative analysis.
Assumption~\ref{ass:conv-filter} records the corresponding
separation on the discrete spectrum in the form used by the graph recurrence.

% The following lemma follows from Assumption~\ref{ass:conv-filter} and the diagonal representation of \(F_h=f(A_h)\) in a \(b\)-orthonormal eigenbasis of \(A_h\).

% \begin{lemma}
% \label{lem:conv-uniform-filter}
% Under Assumption~\ref{ass:conv-filter}, \(F_{h,\mathcal E}\) is invertible and
% \[
%   \|F_{h,\mathcal E}^{-1}\|_b=\alpha_h^{-1},
%   \qquad
%   \|F_{h,\perp}\|_b=\beta_h,
%   \qquad
%   \|F_{h,\perp}\|_b\|F_{h,\mathcal E}^{-1}\|_b=\rho_h<1.
% \]
% \end{lemma}

% \subsection{Approximation of the Initial Subspace to \(\mathcal E_h\)}

\begin{lemma}\label{lem:conv-graph}
Let \( U^k\subset V_h\) be an
\(s\)-dimensional subspace satisfying
\(\theta(U^k,\mathcal E_h)<1\).  Then the restriction
\(P_h|_{ U^k}:U^k\to\mathcal E_h\) is bijective, and \(U^k\) has the graph representation
\[
  U^k
  =
  \{p+G_k^h p:\ p\in\mathcal E_h\},
  \qquad
  G_k^h:=P_h^\perp(P_h|_{U^k})^{-1} .
\]
Moreover,
\[
  \|G_k^h\|_b
  =\tan\Theta(U^k,\mathcal E_h),
  \qquad
  \theta(U^k,\mathcal E_h)
  =
  \frac{\|G_k^h\|_b}{\sqrt{1+\|G_k^h\|_b^2}} .
\]
\end{lemma}

\begin{proof}
By definition, we have
\[
  \theta^k
  := \theta_b(U^k,\mathcal E_h)
  = \sup_{0\ne  u\in U^k}
    \frac{\|P_h^\perp  u\|_b}{\| u\|_b} < 1.
\]
Since \(P_h  u\) and \(P_h^\perp  u\) are \(b\)-orthogonal,
\[
  \|P_h  u\|_b^2
  =
  \|u\|_b^2-\|P_h^\perp  u\|_b^2
  \ge
  (1-(\theta^k)^2)\| u\|_b^2,
  \qquad  u\in U^k.
\]
Thus \(P_h|_{U^k}\) is injective; moreover,
\(\dim U^k=\dim\mathcal E_h=s\), so it is a bijection from
\(U^k\) onto \(\mathcal E_h\).
For \(p\in\mathcal E_h\), let \( u=(P_h|_{ U^k})^{-1}p\). The orthogonal decomposition of
\( u\) yields
\begin{equation}\label{equ:decom_u}
   u=P_h  u+P_h^\perp  u=p+G_k^hp.    
\end{equation}
The arbitrariness of \( u\in U^k\) yields the graph representation of \(U^k\). By \eqref{equ:decom_u} and the bijectivity of \(P_h|_{U^k}\), we obtain
\begin{equation}\label{equ:G_k prop}
  \|G_k^h\|_b
  \sup_{0\ne p\in\mathcal E_h}
  \frac{\|G_k^hp\|_b}{\|p\|_b}
  =
  \sup_{0\ne  u\in U^k}
  \frac{\|P_h^\perp  u\|_b}{\|P_hu\|_b}
  =
  \frac{\theta^k}{\sqrt{1-(\theta^k)^2}}
  =
  \tan\Theta( U^k,\mathcal E_h).
\end{equation}
Similarly, by rewriting \eqref{equ:G_k prop}, we obtain an expression of
\(\theta^k\) in terms of \(\|G_h^k\|_b\).
\end{proof}

\begin{theorem}\label{thm:conv-exact-subspace}
Under Assumption~\ref{ass:conv-filter}, if \(\delta_H:=CH^2<1\),
then the filtered iteration
\[
  U^{k+1}:=F_h(U^k)
  :=\{F_hv:\ v\in U^k\}
\]
is well defined as an \(s\)-dimensional graph over \(\mathcal E_h\)
for all \(k\ge0\). Moreover,
\[
  \|G_k^h\|_b
  \le \rho_h^k\eta_H,
  \qquad
  \theta(U^k,\mathcal E_h)
  \le \rho_h^k\eta_H,
  \qquad
  \eta_H:=
  \frac{\delta_H}{\sqrt{1-\delta_H^2}},
\]
where the contraction factor \(\rho_h\) depends only on the separation between the spectral region of interest and the remainder of the spectrum, together with the coarse mesh size \(H\).
\end{theorem}

\begin{proof}
By Theorem~\ref{theo:T_and_Th}, we have
\[
  \theta_b(\mathcal E_h,\mathcal E)
  \le Ch^2,
  \qquad
  \theta_b(\mathcal E_H,\mathcal E)
  \le CH^2.
\]
Since \(U^0=\mathcal E_H\), the triangle inequality and \(h\le H\)
yield
\[
  \theta(U^0,\mathcal E_h)
  \le
  \theta(\mathcal E_H,\mathcal E)
  +\theta(\mathcal E,\mathcal E_h)
  \le CH^2.
\]
Thus, \(\delta_H<1\) and \cref{lem:conv-graph} imply that \(U^0\)
is a graph over \(\mathcal E_h\), with
\[
  \|G_0^h\|_b
  \le
  \frac{\delta_H}{\sqrt{1-\delta_H^2}}
  =\eta_H.
\]

Suppose that
\[
  U^k=\{p+G_k^hp:\ p\in\mathcal E_h\}.
\]
The invariance of \(\mathcal E_h\) and \(\mathcal E_h^\perp\)
under \(F_h\), together with the bijectivity of
\(F_{h,\mathcal E}\), gives
\[
  U^{k+1}
  =
  \{p+G_{k+1}^hp:\ p\in\mathcal E_h\},
  \qquad
  G_{k+1}^h
  =
  F_{h,\perp}G_k^hF_{h,\mathcal E}^{-1}.
\]
% Hence \(U^{k+1}\) has dimension \(s\), and
% \cref{lem:conv-uniform-filter} gives
% \[
%   \|G_{k+1}^h\|_b
%   \le \rho_h\|G_k^h\|_b.
% \]
Hence \(U^{k+1}\) has dimension \(s\). By the definitions of
\(\alpha_h\) and \(\beta_h\),
\[
  \|F_{h,\mathcal E}^{-1}\|_b=\alpha_h^{-1},
  \qquad
  \|F_{h,\perp}\|_b=\beta_h.
\]
Therefore,
\[
  \|G_{k+1}^h\|_b
  \le
  \|F_{h,\perp}\|_b
  \|G_k^h\|_b
  \|F_{h,\mathcal E}^{-1}\|_b
  \le
  \rho_h\|G_k^h\|_b.
\]
Iterating this inequality yields
\(\|G_k^h\|_b\le\rho_h^k\eta_H\), and hence
\[
  \theta(U^k,\mathcal E_h)
  =
  \frac{\|G_k^h\|_b}{\sqrt{1+\|G_k^h\|_b^2}}
  \le \|G_k^h\|_b
  \le \rho_h^k\eta_H.
\]
\end{proof}

% \subsection{Rayleigh--Ritz extraction and eigenpair errors}

Rayleigh--Ritz extraction on \(U^k\) computes
Ritz pairs \((\lambda_i^k,u_i^k)\in\mathbb R\times U^k\) from \eqref{eq:RR}. This allows different target eigenpairs to be accelerated separately. The following theorem gives error bounds for the Ritz values
and Ritz vectors after incorporating the Rayleigh--Ritz procedure into the iteration.
\begin{theorem}\label{thm:conv-exact-eigenpair}
Under the assumptions stated in \cref{thm:conv-exact-subspace}, let
\(\lambda_i^h\) be a simple target eigenvalue with matched Ritz pair
\((\lambda_i^k,u_i^k)\). Let \(P_k\) be the \(b\)-orthogonal projector
onto \(\mathcal U^k\), and define
\(
  \delta_i^{(k)}
  =
  \min_{\substack{1\le j\le s\\ j\ne i}}
  |\lambda_i^h-\lambda_j^k|,
  \
  \beta^{(k)}
  =
  \|P_kA_h(I-P_k)\|_b .
\)
If \(\delta_i^{(k)}>0\), define
\[
  \kappa_i^{(k)}
  =
  \left(1+\frac{(\beta^{(k)})^2}{(\delta_i^{(k)})^2}\right)^{1/2}.
\]
Then
\begin{align*}
  \theta(u_i^k,u_i^h)
  &\le
  \kappa_i^{(k)}\eta_H\rho_h^k, \\
  |\lambda_i^k-\lambda_i^h|
  &\le
  \|A_h-\lambda_i^hI\|_b
  \bigl(\kappa_i^{(k)}\eta_H\bigr)^2\rho_h^{2k}.
\end{align*}
\end{theorem}

\begin{proof}
By \cref{thm:conv-exact-subspace},
\[
  \theta(U^k,\mathcal E_h)\le\eta_H\rho_h^k .
\]
Since \(u_i^h\in\mathcal E_h\) and \(P_k\) is the \(b\)-orthogonal
projector onto \(U^k\),
\[
  \|(I-P_k)u_i^h\|_b
  \le
  \theta(U^k,\mathcal E_h)
  \le
  \eta_H\rho_h^k .
\]
The Rayleigh--Ritz perturbation estimate for a self-adjoint operator in
the \(b\) inner product \cite[Chapter~4]{saad2011numerical} gives
\[
  \theta(u_i^k,u_i^h)
  \le
  \kappa_i^{(k)}\|(I-P_k)u_i^h\|_b.
\]
Substituting the preceding subspace bound proves the eigenvector
estimate.
The eigenvalue estimate follows directly from \cite[Proposition~4.5]{saad2011numerical}.
% For the eigenvalue estimate, use the Rayleigh quotient identity
% \(\lambda_i^k=a(u_i^k,u_i^k)\) with \(b(u_i^k,u_i^k)=1\).  After
% choosing the phase of \(u_i^k\), write
% \[
%   u_i^k=\alpha u_i^h+w,\qquad
%   w\perp u_i^h,\qquad
%   \|w\|_b=\theta(u_i^k,u_i^h).
% \]
% The variational eigenvalue relation for \(u_i^h\) and the
% \(b\)-self-adjointness of \(A_h\) eliminate the cross terms, giving
% \[
%   \lambda_i^k-\lambda_i^h
%   =
%   b((A_h-\lambda_i^hI)w,w).
% \]
% Therefore
% \[
%   |\lambda_i^k-\lambda_i^h|
%   \le
%   \|A_h-\lambda_i^hI\|_b\|w\|_b^2 .
% \]
% Inserting the eigenvector estimate gives the asserted eigenvalue bound.
\end{proof}

\section{Numerical Experiment}\label{sec:applications}

In this section, we use the linear elasticity model to evaluate the performance and robustness of the TLFSI method. 
The method is also applicable to other elliptic differential eigenvalue problems.
% The computational environment, model equation, and solver parameters are specified below.

\subsection{Setup}

All numerical experiments reported in this work were conducted on the High-Performance Computing (HPC) platform of Xiangtan University. The system runs the 64-bit Linux CentOS 7.5 operating system and consists of 85 compute nodes. Each node is equipped with two Intel Xeon Platinum 8358 processors operating at 2.60 GHz, providing a total of 64 CPU cores and 512 GB of memory per node. The theoretical double-precision peak performance of the entire platform is approximately 452608 GFlops. All numerical experiments in this work were performed in serial mode. The codes were compiled with the `-O2` optimization option.

The TLFSI method is applied to linear elasticity interior eigenvalue problems under different boundary conditions. For comparison, we also test the FEAST method \cite{Po09,TaPo14}, using FEAST version 4.0, and our implementation of the PHSI method \cite{LiXu23}. 
Each comparison uses the same finite element matrices and linear algebra kernels whenever the algorithms share these components.

% FEAST and TLFSI use eight quadrature nodes for the contour rule. 
The subspace dimension \(s\) is set to
\(
s=\lceil 1.8\tilde e\rceil,
\)
where \(\tilde e\) is the estimated number of eigenvalues in the interval.
Since the cost of estimating the number of eigenvalues is common to all methods and no corresponding routine is available in FEAST, this cost is excluded from the reported timings in all performance tests, although it is negligible in practice.
The shifted linear systems are solved using the sparse direct solver MUMPS version 4.10.0 \cite{amestoy2001fully}, with at most \(50\) outer iterations. For every method, the outer iteration will stop when the maximum absolute algebraic residual satisfies
\[
  \max_{1\leq i\leq s}
  \{r_i^k\}\leq10^{-6},
\]
where
\[
r_i^k=
\frac{\left\|A v_i^k-\lambda_i B v_i^k\right\|_2}
{\left\|\alpha B v_i^k\right\|_2},
\qquad
\alpha=\max(|a|,|b|),
\]
where \(A\) and \(B\) are the stiffness and mass matrices associated with \(a(\cdot,\cdot)\) and \(b(\cdot,\cdot)\), respectively.
The convergence criterion follows \cite{polizzi2020feast}. Note that, for the linear elasticity problem, we always have \(\alpha=b\).
Throughout the numerical experiments, we set \(\uprho=1, E=3000\) and \(\upsilon=0.3\) in \eqref{eq:abstract-pde-eig}.  
The FEAST and TLFSI method use the same filter parameters, subspace dimension, shifted-system tolerance, and stopping criterion specified above. 
For PHSI, we set \(\alpha=1.0\). For a fair comparison, the additive Schwarz preconditioner is replaced by a direct solver.
% The reported data include the contour, eigenvalue-count estimate, initial
% subspace dimension, inner shifted solver, preconditioner update policy,
% and maximum inner iteration count used in each experiment.

\subsection{Quality of the Initial Subspace}

According to Theorem \ref{theo:T_and_Th}, the initial subspace generated by the coarse space gradually approaches the target invariant subspace as the coarse mesh size $H$ decreases. We further verify this theoretical result through several representative numerical experiments, where the random initial subspace is also included as a comparison.

For the test reported in Table~\ref{tab:subspace_angle}, we consider the linear elasticity eigenvalue problem on \(\Omega=[0,1]^3\) with pure Neumann boundary conditions and discretize it using \(P_1\) finite elements on uniform meshes. The table reports the maximum subspace angles for the target spectral interval \([9.0\times10^{3},3.0\times10^{4}]\) on three fine meshes, \(h=\frac{1}{12}\), \(h=\frac{1}{24}\), and \(h=\frac{1}{48}\). The corresponding initial subspaces are constructed either from Gaussian random vectors or from coarse spaces with different mesh sizes \(H\). The quantity \(\sin\theta_{\max}\) denotes the sine of the maximum angle between the target eigenspace and the initial subspace. For the coarse-grid initial subspaces, we also report the observed convergence orders with respect to \(H\), computed relative to the reference coarse mesh \(H=\frac{1}{3}\).

\begin{table}[htbp]
\centering
\caption{Maximum subspace angles and corresponding convergence orders}
\label{tab:subspace_angle}
\begingroup\small\setlength{\tabcolsep}{4pt}
\begin{tabular}{cccccccccc}
\toprule
\multirow{2}{*}{$1/h$} & \multirow{2}{*}{Item} & \multirow{2}{*}{Random} & \multicolumn{7}{c}{$1/H$} \tabularnewline
\cmidrule(lr){4-10}
& & & 3 & 4 & 6 & 8 & 12 & 16 & 24 \tabularnewline
\midrule
\multirow{2}{*}{12} & $\sin\theta_{\max}$ & 9.98E-01 & 4.50E-01 & 2.89E-01 & 1.22E-01 & -- & -- & -- & -- \tabularnewline
& Order & -- & -- & 1.54 & 1.88 & -- & -- & -- & -- \tabularnewline
\midrule
\multirow{2}{*}{24} & $\sin\theta_{\max}$ & 1.00E+00 & 4.73E-01 & 3.15E-01 & 1.51E-01 & 9.06E-02 & 4.13E-02 & -- & -- \tabularnewline
& Order & -- & -- & 1.41 & 1.64 & 1.69 & 1.76 & -- & -- \tabularnewline
\midrule
\multirow{2}{*}{48} & $\sin\theta_{\max}$ & 1.00E+00 & 4.79E-01 & 3.22E-01 & 1.59E-01 & 1.03E-01 & 5.37E-02 & 3.13E-02 & 1.27E-02 \tabularnewline
& Order & -- & -- & 1.38 & 1.59 & 1.57 & 1.58 & 1.63 & 1.75 \tabularnewline
\bottomrule
\end{tabular}
\endgroup
\end{table}

As shown in Table~\ref{tab:subspace_angle}, the values of \(\sin\theta_{\max}\) for the random initial subspaces remain close to \(1\) on all three fine meshes, indicating that some directions in the target eigenspace are nearly orthogonal to the randomly generated initial subspace. In contrast, the coarse-grid initial subspaces provide substantially better approximations even for relatively coarse values of \(H\). Moreover, \(\sin\theta_{\max}\) decreases consistently as \(H\) is refined. The observed convergence rates are generally around \(1.4\)--\(1.8\), exhibiting an overall empirical behavior close to \(\mathcal{O}(H^{1.6})\). These results indicate that the coarse-grid initial subspace converges steadily toward the target eigenspace and provides a markedly more accurate initial approximation than the Gaussian random subspace.

To further investigate the influence of the coarse mesh size \(H\) on the convergence and computational efficiency of the proposed TLFSI method, we perform numerical experiments for several fine mesh sizes \(h\) and a range of admissible coarse mesh sizes \(H\). The model problem, boundary conditions, and physical parameters are the same as those used in Table~\ref{tab:subspace_angle}. The fine and coarse meshes are required to be nested, namely, each coarse-grid element is composed of an integer number of fine-grid elements in each coordinate direction. Unless otherwise stated, all numerical experiments in the following tests use \(n_q=4\) quadrature nodes on the contour. By exploiting the conjugate symmetry in~\eqref{eq:real-filter}, only \(n_q/2=2\) complex shifted systems need to be solved for each filter application.

Under this setting, Table~\ref{tab:TLFSI_INT_IDX} reports the numerical performance of TLFSI-INT and TLFSI-IDX for four target spectral intervals and different combinations of \(h\) and \(H\). For each pair \((h,H)\), we report the number of outer iterations, the number of shifted matrix solves, and the total computational time. Here, \(e\) denotes the number of eigenvalues contained in the target interval, while \(s\) denotes the dimension of the initial subspace.

\begin{table}[htbp]
\centering
\captionsetup{justification=centering}
\caption{Numerical performance of TLFSI-INT and TLFSI-IDX for different fine mesh sizes $h$ and coarse mesh sizes $H$}
\label{tab:TLFSI_INT_IDX}
\begingroup
\scriptsize
\setlength{\tabcolsep}{6pt}
\begin{tabular}{cccccccccc}
\toprule
% ============================================================
% h = 1/24
% ============================================================
\multicolumn{10}{c}{$h=1/24$}\\
\midrule
\multirow{4}{*}{$H$} &
\multirow{4}{*}{Method} &
\multicolumn{2}{c}{$[9.0e+3,3.0e+4]$} &
\multicolumn{2}{c}{$[6.0e+4,9.0e+4]$} &
\multicolumn{2}{c}{$[9.0e+4,1.2e+5]$} &
\multicolumn{2}{c}{$[1.2e+5,1.5e+5]$}\\
\cmidrule(lr){3-4}\cmidrule(lr){5-6}
\cmidrule(lr){7-8}\cmidrule(lr){9-10}
&&\multicolumn{2}{c}{$e=17,\ s=26$}
&\multicolumn{2}{c}{$e=22,\ s=43$}
&\multicolumn{2}{c}{$e=27,\ s=45$}
&\multicolumn{2}{c}{$e=23,\ s=39$}\\
&&Iter. (Solves)&Time
&Iter. (Solves)&Time
&Iter. (Solves)&Time
&Iter. (Solves)&Time\\
\midrule
\multirow{2}{*}{$1/3$}
&TLFSI-INT
&3 (196)&15.47
&6 (468)&31.55
&6 (458)&30.85
&11 (602)&38.20\\
&TLFSI-IDX
&3 (196)&17.43
&7 (482)&33.18
&5 (410)&29.54
&10 (570)&38.17\\
\cmidrule(lr){1-10}

\multirow{2}{*}{$1/4$}
&TLFSI-INT
&3 (196)&15.54
&6 (448)&29.83
&6 (454)&31.32
&11 (604)&39.39\\
&TLFSI-IDX
&3 (196)&17.21
&4 (346)&26.20
&4 (372)&27.05
&10 (602)&39.43\\
\cmidrule(lr){1-10}
\multirow{2}{*}{$1/6$}
&TLFSI-INT
&3 (194)&15.92
&3 (298)&22.13
&4 (348)&26.04
&10 (574)&37.03\\
&TLFSI-IDX
&3 (194)&17.25
&3 (298)&23.73
&4 (356)&27.81
&4 (330)&26.00\\
\cmidrule(lr){1-10}
\multirow{2}{*}{$1/8$}
&TLFSI-INT
&3 (186)&15.18
&3 (298)&21.45
&4 (350)&25.59
&7 (432)&29.76\\
&TLFSI-IDX
&3 (186)&17.37
&3 (298)&23.78
&4 (360)&27.90
&4 (322)&25.50\\
\cmidrule(lr){1-10}
\multirow{2}{*}{$1/12$}
&TLFSI-INT
&3 (172)&15.24
&3 (284)&22.06
&3 (304)&23.72
&3 (274)&21.70\\
&TLFSI-IDX
&3 (172)&22.01
&3 (284)&28.49
&4 (360)&34.99
&4 (310)&33.14\\
\midrule
% ============================================================
% h = 1/48
% ============================================================
\multicolumn{10}{c}{$h=1/48$}\\
\midrule
\multirow{4}{*}{$H$} &
\multirow{4}{*}{Method} &
\multicolumn{2}{c}{$[9.0e+3,3.0e+4]$} &
\multicolumn{2}{c}{$[6.0e+4,9.0e+4]$} &
\multicolumn{2}{c}{$[9.0e+4,1.2e+5]$} &
\multicolumn{2}{c}{$[1.2e+5,1.5e+5]$}\\
\cmidrule(lr){3-4}\cmidrule(lr){5-6}
\cmidrule(lr){7-8}\cmidrule(lr){9-10}
&&\multicolumn{2}{c}{$e=17,\ s=25$}
&\multicolumn{2}{c}{$e=27,\ s=44$}
&\multicolumn{2}{c}{$e=23,\ s=41$}
&\multicolumn{2}{c}{$e=22,\ s=32$}\\
&&Iter. (Solves)&Time
&Iter. (Solves)&Time
&Iter. (Solves)&Time
&Iter. (Solves)&Time\\
\midrule
\multirow{2}{*}{$1/3$}
&TLFSI-INT
&3 (190)&220.24
&7 (526)&434.28
&11 (636)&498.67
&11 (502)&412.60\\
&TLFSI-IDX
&3 (190)&243.92
&11 (608)&521.85
&6 (436)&407.89
&11 (526)&450.63\\
\cmidrule(lr){1-10}
\multirow{2}{*}{$1/4$}
&TLFSI-INT
&3 (190)&216.26
&10 (584)&491.68
&5 (384)&347.07
&15 (584)&477.52\\
&TLFSI-IDX
&3 (190)&243.45
&4 (348)&360.03
&4 (360)&360.70
&12 (588)&484.58\\
\cmidrule(lr){1-10}
\multirow{2}{*}{$1/6$}
&TLFSI-INT
&3 (188)&220.50
&3 (304)&291.59
&5 (378)&330.30
&14 (620)&482.76\\
&TLFSI-IDX
&3 (188)&243.37
&3 (304)&326.17
&5 (390)&376.21
&8 (392)&365.77\\
\cmidrule(lr){1-10}
\multirow{2}{*}{$1/8$}
&TLFSI-INT
&3 (188)&211.40
&3 (304)&289.14
&6 (406)&356.64
&11 (504)&406.15\\
&TLFSI-IDX
&3 (188)&247.51
&3 (304)&313.78
&4 (342)&343.46
&7 (354)&347.70\\
\cmidrule(lr){1-10}
\multirow{2}{*}{$1/12$}
&TLFSI-INT
&3 (174)&210.88
&3 (300)&281.25
&6 (406)&353.37
&9 (432)&376.80\\
&TLFSI-IDX
&3 (174)&248.79
&3 (300)&331.36
&4 (340)&347.21
&7 (338)&350.37\\
\cmidrule(lr){1-10}
\multirow{2}{*}{$1/16$}
&TLFSI-INT
&3 (166)&214.97
&3 (292)&283.33
&3 (282)&275.87
&10 (438)&379.47\\
&TLFSI-IDX
&3 (166)&251.68
&3 (292)&332.24
&4 (328)&365.90
&7 (368)&387.95\\
\cmidrule(lr){1-10}
\multirow{2}{*}{$1/24$}
&TLFSI-INT
&3 (158)&210.45
&3 (278)&303.05
&4 (304)&317.95
&4 (258)&276.92\\
&TLFSI-IDX
&3 (158)&361.03
&3 (278)&455.66
&4 (320)&499.56
&4 (256)&479.63\\
\bottomrule
\end{tabular}
\endgroup
\end{table}

The results in Table~\ref{tab:TLFSI_INT_IDX} show that the coarse mesh size \(H\) has a significant influence on both convergence and computational cost. In general, reducing \(H\) improves the quality of the coarse-grid initial subspace and tends to decrease the number of outer iterations and shifted matrix solves, particularly for the higher spectral intervals. However, a finer coarse mesh also increases the cost of the coarse-grid eigenvalue problem. Therefore, the smallest admissible value of \(H\) does not necessarily yield the shortest total computational time.

Following the strategy discussed in Section~\ref{sec:select_H}, we determine \(H\) according to criterion~\eqref{equ:compute_H}. For the present problem,
\[
\mu_L=\frac{E}{2(1+\nu)}=1153.846.
\]
Among all coarse meshes satisfying the mesh-nesting condition, we choose the largest \(H\) satisfying
\[
H := H(b,\tau,h)
=
\max\left\{
K\in S_h:
K<\frac{\tau}{\kappa_S(b)}
\right\},
\]
where
\[
S_h=
\left\{
\frac{1}{m}:\,
m\in\mathbb N,\,
m\mid\frac{1}{h}
\right\}.
\]

The numerical results suggest that \(\tau=2.0\) provides a suitable balance between the approximation quality of the coarse-grid initial subspace and the cost of the coarse problem. With this choice,
\[
H=
\begin{cases}
1/3, & b=3.0\times10^4,\\
1/6, & b=9.0\times10^4,\\
1/6, & b=1.2\times10^5,\\
1/6, & b=1.5\times10^5.
\end{cases}
\]

This selection is consistent with the observed numerical behavior. For example, for \(h=1/48\) and the interval \([6.0\times10^4,9.0\times10^4]\), TLFSI-INT requires \(10\) iterations and \(491.68\) seconds for \(H=1/4\), whereas \(H=1/6\) requires only \(3\) iterations and \(291.59\) seconds. TLFSI-IDX exhibits the same tendency: the iteration count decreases from \(4\) to \(3\), while the computational time decreases from \(360.03\) to \(326.17\) seconds. These results indicate that an excessively coarse initial space may noticeably deteriorate convergence, whereas further refinement beyond a suitable value of \(H\) does not consistently reduce the total computational time.

To further examine the robustness of the choice \(\tau=2.0\), we consider three additional and wider target spectral intervals,
\[
[9.0\times10^3,6.0\times10^4],\qquad
[6.0\times10^4,1.2\times10^5],\qquad
[1.2\times10^5,1.8\times10^5].
\]
These intervals differ from those considered above and therefore provide an additional test of whether the same coarse-mesh selection criterion remains effective without retuning \(\tau\). Table~\ref{tab:H_validation} reports the corresponding numerical results for both TLFSI-INT and TLFSI-IDX.

\begin{table}[htbp]
\centering
\caption{Numerical performance of TLFSI-INT and TLFSI-IDX for additional target spectral intervals}
\label{tab:H_validation}
\begingroup
\scriptsize
\setlength{\tabcolsep}{6pt}

\begin{tabular}{cccccccc}
\toprule

% ============================================================
% h = 1/24
% ============================================================
\multicolumn{8}{c}{$h=1/24$}\\
\midrule

\multirow{4}{*}{$H$} &
\multirow{4}{*}{Method} &
\multicolumn{2}{c}{$[9.0e+3,6.0e+4]$} &
\multicolumn{2}{c}{$[6.0e+4,1.2e+5]$} &
\multicolumn{2}{c}{$[1.2e+5,1.8e+5]$}\\
\cmidrule(lr){3-4}\cmidrule(lr){5-6}\cmidrule(lr){7-8}

&&\multicolumn{2}{c}{$e=33,\ s=51$}
&\multicolumn{2}{c}{$e=49,\ s=73$}
&\multicolumn{2}{c}{$e=50,\ s=87$}\\

&&Iter. (Solves)&Time
&Iter. (Solves)&Time
&Iter. (Solves)&Time\\
\midrule

\multirow{2}{*}{$1/3$}
&TLFSI-INT
&4 (392)&27.78
&4 (584)&39.25
&11 (1304)&83.80\\
&TLFSI-IDX
&4 (392)&29.77
&11 (966)&62.40
&4 (702)&48.54\\
\cmidrule(lr){1-8}

\multirow{2}{*}{$1/4$}
&TLFSI-INT
&4 (388)&27.92
&4 (586)&40.78
&5 (774)&53.68\\
&TLFSI-IDX
&4 (388)&29.86
&5 (632)&45.13
&4 (684)&47.62\\
\cmidrule(lr){1-8}

\multirow{2}{*}{$1/6$}
&TLFSI-INT
&3 (346)&24.59
&4 (558)&39.19
&5 (762)&53.97\\
&TLFSI-IDX
&3 (346)&26.02
&4 (564)&41.69
&4 (692)&48.18\\
\cmidrule(lr){1-8}

\multirow{2}{*}{$1/8$}
&TLFSI-INT
&3 (340)&24.47
&4 (562)&39.39
&4 (642)&46.72\\
&TLFSI-IDX
&3 (340)&26.10
&4 (562)&42.26
&4 (684)&50.61\\
\cmidrule(lr){1-8}

\multirow{2}{*}{$1/12$}
&TLFSI-INT
&3 (334)&25.46
&4 (550)&39.74
&3 (562)&42.27\\
&TLFSI-IDX
&3 (334)&31.87
&4 (550)&48.83
&4 (670)&59.04\\

\midrule

% ============================================================
% h = 1/48
% ============================================================
\multicolumn{8}{c}{$h=1/48$}\\
\midrule

\multirow{4}{*}{$H$} &
\multirow{4}{*}{Method} &
\multicolumn{2}{c}{$[9.0e+3,6.0e+4]$} &
\multicolumn{2}{c}{$[6.0e+4,1.2e+5]$} &
\multicolumn{2}{c}{$[1.2e+5,1.8e+5]$}\\
\cmidrule(lr){3-4}\cmidrule(lr){5-6}\cmidrule(lr){7-8}

&&\multicolumn{2}{c}{$e=33,\ s=62$}
&\multicolumn{2}{c}{$e=50,\ s=73$}
&\multicolumn{2}{c}{$e=57,\ s=102$}\\

&&Iter. (Solves)&Time
&Iter. (Solves)&Time
&Iter. (Solves)&Time\\
\midrule

\multirow{2}{*}{$1/3$}
&TLFSI-INT
&4 (478)&418.34
&6 (746)&610.22
&7 (1092)&822.90\\
&TLFSI-IDX
&4 (478)&443.88
&12 (1032)&788.76
&5 (886)&752.80\\
\cmidrule(lr){1-8}

\multirow{2}{*}{$1/4$}
&TLFSI-INT
&4 (472)&416.73
&6 (702)&568.78
&4 (780)&630.68\\
&TLFSI-IDX
&4 (472)&429.22
&5 (648)&574.47
&4 (774)&677.81\\
\cmidrule(lr){1-8}

\multirow{2}{*}{$1/6$}
&TLFSI-INT
&3 (408)&361.59
&4 (560)&460.48
&5 (888)&705.46\\
&TLFSI-IDX
&3 (408)&378.02
&4 (572)&492.54
&4 (778)&670.18\\
\cmidrule(lr){1-8}

\multirow{2}{*}{$1/8$}
&TLFSI-INT
&3 (404)&352.03
&7 (716)&579.20
&4 (750)&631.35\\
&TLFSI-IDX
&3 (404)&390.71
&7 (716)&613.48
&4 (774)&654.41\\
\cmidrule(lr){1-8}

\multirow{2}{*}{$1/12$}
&TLFSI-INT
&3 (400)&350.20
&6 (660)&560.57
&4 (744)&618.68\\
&TLFSI-IDX
&3 (400)&399.36
&6 (660)&587.01
&4 (770)&680.82\\
\cmidrule(lr){1-8}

\multirow{2}{*}{$1/16$}
&TLFSI-INT
&3 (386)&348.02
&5 (594)&501.27
&3 (652)&548.06\\
&TLFSI-IDX
&3 (386)&392.75
&5 (594)&556.79
&4 (748)&676.59\\
\cmidrule(lr){1-8}

\multirow{2}{*}{$1/24$}
&TLFSI-INT
&3 (362)&353.39
&4 (546)&492.10
&3 (640)&566.92\\
&TLFSI-IDX
&3 (362)&514.05
&4 (546)&668.35
&3 (642)&812.56\\

\bottomrule
\end{tabular}

\endgroup
\end{table}

According to the criterion with \(\tau=2.0\), the coarse mesh sizes corresponding to the three upper spectral bounds are
\[
H=
\begin{cases}
1/4, & b=6.0\times10^4,\\
1/6, & b=1.2\times10^5,\\
1/8, & b=1.8\times10^5.
\end{cases}
\]

The results in Table~\ref{tab:H_validation} further support the choice \(\tau=2.0\). For the interval \([6.0\times10^4,1.2\times10^5]\) and \(h=1/48\), the selected value \(H=1/6\) reduces the iteration count and computational time compared with the coarser choice \(H=1/4\) for both TLFSI-INT and TLFSI-IDX. Similarly, for the higher interval \([1.2\times10^5,1.8\times10^5]\), the selected value \(H=1/8\) improves the performance compared with \(H=1/6\). For the lower interval \([9.0\times10^3,6.0\times10^4]\), although finer coarse meshes may yield slightly smaller computational times in some cases, the selected value \(H=1/4\) already provides stable convergence.

Overall, the results for both the original and additional spectral intervals show that \(\tau=2.0\), without interval-dependent retuning, provides a reasonable and robust balance between avoiding an overly coarse initial space and limiting the cost of the coarse-grid eigenvalue problem. We therefore use \(\tau=2.0\) in the subsequent experiments.

Finally, we examine the effect of the deflation technique using the coarse mesh sizes selected above by the criterion with \(\tau=2.0\). Specifically, \(H=1/4\), \(H=1/6\), and \(H=1/8\) are used for the target intervals \([9.0\times10^3,6.0\times10^4]\), \([6.0\times10^4,1.2\times10^5]\), and \([1.2\times10^5,1.8\times10^5]\), respectively. Table~\ref{tab:deflation_comparison} compares the default TLFSI variants with their counterparts without deflation for \(h=1/24\) and \(h=1/48\). Here, the suffix ``-ND'' denotes that the deflation technique is not used, while all other algorithmic settings are kept unchanged.

\begin{table}[htbp]
\centering
\caption{Comparison of TLFSI with and without deflation using the coarse mesh sizes selected by \(\tau=2.0\)}
\label{tab:deflation_comparison}
\begingroup
\scriptsize
\setlength{\tabcolsep}{6pt}

\begin{tabular}{ccccccc}
\toprule

% ============================================================
% h = 1/24
% ============================================================
\multicolumn{7}{c}{$h=1/24$}\\
\midrule

\multirow{3}{*}{Method} &
\multicolumn{2}{c}{$[9.0e+3,6.0e+4]$} &
\multicolumn{2}{c}{$[6.0e+4,1.2e+5]$} &
\multicolumn{2}{c}{$[1.2e+5,1.8e+5]$}\\
\cmidrule(lr){2-3}\cmidrule(lr){4-5}\cmidrule(lr){6-7}

&\multicolumn{2}{c}{$e=33,\ s=51$}
&\multicolumn{2}{c}{$e=49,\ s=73$}
&\multicolumn{2}{c}{$e=50,\ s=87$}\\

&Iter. (Solves)&Time
&Iter. (Solves)&Time
&Iter. (Solves)&Time\\
\midrule

TLFSI-INT-ND
&4 (448)&30.59
&4 (624)&42.75
&4 (736)&50.18\\

TLFSI-INT
&4 (388)&27.92
&4 (558)&39.19
&4 (642)&46.72\\

TLFSI-IDX-ND
&4 (448)&31.59
&4 (624)&43.80
&4 (736)&52.25\\

TLFSI-IDX
&4 (388)&29.86
&4 (564)&41.69
&4 (684)&50.61\\

\midrule

% ============================================================
% h = 1/48
% ============================================================
\multicolumn{7}{c}{$h=1/48$}\\
\midrule

\multirow{3}{*}{Method} &
\multicolumn{2}{c}{$[9.0e+3,6.0e+4]$} &
\multicolumn{2}{c}{$[6.0e+4,1.2e+5]$} &
\multicolumn{2}{c}{$[1.2e+5,1.8e+5]$}\\
\cmidrule(lr){2-3}\cmidrule(lr){4-5}\cmidrule(lr){6-7}

&\multicolumn{2}{c}{$e=33,\ s=62$}
&\multicolumn{2}{c}{$e=50,\ s=73$}
&\multicolumn{2}{c}{$e=57,\ s=102$}\\

&Iter. (Solves)&Time
&Iter. (Solves)&Time
&Iter. (Solves)&Time\\
\midrule

TLFSI-INT-ND
&4 (536)&419.33
&4 (624)&505.05
&4 (856)&663.06\\

TLFSI-INT
&4 (472)&416.73
&4 (560)&460.48
&4 (750)&631.35\\

TLFSI-IDX-ND
&4 (536)&466.40
&4 (624)&526.41
&4 (856)&690.83\\

TLFSI-IDX
&4 (472)&429.22
&4 (572)&492.54
&4 (774)&654.41\\

\bottomrule
\end{tabular}

\endgroup
\end{table}

As shown in Table~\ref{tab:deflation_comparison}, deflation does not change the number of outer iterations, but consistently reduces both the number of matrix solves and the total computational time. For example, for \(h=1/48\) and the interval \([6.0\times10^4,1.2\times10^5]\), the matrix solves decrease from \(624\) to \(560\) for TLFSI-INT and from \(624\) to \(572\) for TLFSI-IDX, while the corresponding computational times decrease from \(505.05\) to \(460.48\) seconds and from \(526.41\) to \(492.54\) seconds, respectively. Similar reductions are observed for the other test cases. Therefore, unless otherwise stated, TLFSI-INT and TLFSI-IDX employ deflation by default in all subsequent numerical experiments.

\subsection{Performance}
\label{sec:exp-elasticity}

We evaluate the performance and adaptability of the proposed TLFSI method for the linear elasticity eigenvalue problem~\eqref{eq:abstract-pde-eig} under both pure Neumann and Robin boundary conditions. To assess the effectiveness of the proposed framework under different boundary conditions, we compare TLFSI with a randomly initialized variant, TLFSI-RAN, and with FEAST. The two coarse-grid initialization strategies introduced above, TLFSI-INT and TLFSI-IDX, are both included in the comparison. Here, TLFSI-RAN uses an initial subspace generated by random vectors whose entries are independently drawn from the standard normal distribution. The physical parameters and finite element discretization are kept consistent with the preceding experiments, and all TLFSI variants employ the deflation technique by default. For TLFSI-INT and TLFSI-IDX, the coarse mesh sizes are selected using the criterion with \(\tau=2.0\), while TLFSI-RAN uses a random initial subspace of the same dimension \(s\). Tables~\ref{tab:TLFSI_random_FEAST} and~\ref{tab:TLFSI_random_FEAST_Robin} report the corresponding numerical results for the Neumann and Robin boundary conditions, respectively. To ensure a fair comparison with FEAST, the computational times reported for the TLFSI variants in these two tables do not include the cost of estimating the number of eigenvalues in the target interval. This is because FEAST requires the user to prescribe the dimension of the search subspace in advance, rather than estimating the eigenvalue count as part of the eigensolver. In both tables, \(e\) denotes the number of eigenvalues contained in the target interval, and \(s\) denotes the dimension of the search subspace.

\begin{table}[htbp]
\centering
\caption{Comparison of TLFSI with different initial subspaces and FEAST for pure Neumann boundary conditions}
\label{tab:TLFSI_random_FEAST}
\begingroup
\scriptsize
\setlength{\tabcolsep}{6pt}

\begin{tabular}{cccc}
\toprule

% ============================================================
% h = 1/24
% ============================================================
\multicolumn{4}{c}{$h=1/24$}\\
\midrule

\multirow{3}{*}{Solver} &
$[9.0e+3,6.0e+4]$ &
$[6.0e+4,1.2e+5]$ &
$[1.2e+5,1.8e+5]$\\

&
$e=33,\ s=51$ &
$e=49,\ s=73$ &
$e=50,\ s=87$\\

&Iterations (Time)
&Iterations (Time)
&Iterations (Time)\\
\midrule

TLFSI-INT
&4 (25.96)
&4 (37.23)
&4 (44.69)\\

TLFSI-IDX
&4 (27.76)
&4 (39.67)
&4 (48.59)\\

TLFSI-RAN
&4 (27.81)
&5 (44.86)
&5 (52.80)\\

FEAST
&5 (30.95)
&6 (51.01)
&6 (59.18)\\

\midrule

% ============================================================
% h = 1/48
% ============================================================
\multicolumn{4}{c}{$h=1/48$}\\
\midrule

\multirow{3}{*}{Solver} &
$[9.0e+3,6.0e+4]$ &
$[6.0e+4,1.2e+5]$ &
$[1.2e+5,1.8e+5]$\\

&
$e=33,\ s=62$ &
$e=50,\ s=73$ &
$e=57,\ s=102$\\

&Iterations (Time)
&Iterations (Time)
&Iterations (Time)\\
\midrule

TLFSI-INT
&4 (392.78)
&4 (438.12)
&4 (607.52)\\

TLFSI-IDX
&4 (406.37)
&4 (470.44)
&4 (628.45)\\

TLFSI-RAN
&4 (401.73)
&5 (529.73)
&5 (727.01)\\

FEAST
&4 (394.07)
&5 (520.37)
&5 (664.23)\\

\bottomrule
\end{tabular}

\endgroup
\end{table}

\begin{table}[htbp]
\centering
\caption{Comparison of TLFSI with different initial subspaces and FEAST for Robin boundary conditions}
\label{tab:TLFSI_random_FEAST_Robin}
\begingroup
\scriptsize
\setlength{\tabcolsep}{6pt}

\begin{tabular}{cccc}
\toprule

% ============================================================
% h = 1/24
% ============================================================
\multicolumn{4}{c}{$h=1/24$}\\
\midrule

\multirow{3}{*}{Solver} &
$[3.0e+4,9.0e+4]$ &
$[9.0e+4,1.5e+5]$ &
$[1.5e+5,2.1e+5]$\\

&
$e=13,\ s=24$ &
$e=26,\ s=36$ &
$e=30,\ s=54$\\

&Iterations (Time)
&Iterations (Time)
&Iterations (Time)\\
\midrule

TLFSI-INT
&3 (12.71)
&6 (21.53)
&3 (24.63)\\

TLFSI-IDX
&3 (13.97)
&4 (20.89)
&4 (29.88)\\

TLFSI-RAN
&4 (14.89)
&5 (22.19)
&4 (28.30)\\

FEAST
&4 (15.16)
&4 (20.05)
&5 (33.80)\\

\midrule

% ============================================================
% h = 1/48
% ============================================================
\multicolumn{4}{c}{$h=1/48$}\\
\midrule

\multirow{3}{*}{Solver} &
$[3.0e+4,9.0e+4]$ &
$[9.0e+4,1.5e+5]$ &
$[1.5e+5,2.1e+5]$\\

&
$e=13,\ s=22$ &
$e=26,\ s=42$ &
$e=37,\ s=58$\\

&Iterations (Time)
&Iterations (Time)
&Iterations (Time)\\
\midrule

TLFSI-INT
&3 (180.72)
&3 (245.71)
&3 (328.79)\\

TLFSI-IDX
&3 (209.45)
&4 (319.85)
&4 (397.73)\\

TLFSI-RAN
&4 (205.25)
&5 (326.48)
&4 (382.96)\\

FEAST
&4 (202.60)
&4 (289.75)
&5 (429.22)\\

\bottomrule
\end{tabular}

\endgroup
\end{table}

The results in Tables~\ref{tab:TLFSI_random_FEAST} and~\ref{tab:TLFSI_random_FEAST_Robin} show that TLFSI maintains stable convergence behavior under both pure Neumann and Robin boundary conditions, demonstrating its adaptability to different boundary settings. Under pure Neumann boundary conditions, the TLFSI variants based on coarse-grid spectral information generally require fewer outer iterations and less computational time than TLFSI-RAN, particularly for the higher spectral intervals. For example, for \(h=1/48\) and the interval \([6.0\times10^4,1.2\times10^5]\), TLFSI-INT and TLFSI-IDX require \(438.12\) and \(470.44\) seconds, respectively, compared with \(529.73\) seconds for TLFSI-RAN and \(520.37\) seconds for FEAST.

A similar behavior is observed under Robin boundary conditions. For \(h=1/48\), TLFSI-INT requires three outer iterations for all three target intervals, while TLFSI-RAN requires four, five, and four iterations, respectively. For the interval \([9.0\times10^4,1.5\times10^5]\), TLFSI-INT requires \(245.71\) seconds, compared with \(326.48\) seconds for TLFSI-RAN and \(289.75\) seconds for FEAST. The TLFSI-IDX results also remain stable over the tested intervals.

Overall, the experiments under the two boundary conditions indicate that the proposed TLFSI framework is not restricted to a particular boundary condition and maintains stable convergence and computational efficiency for both Neumann and Robin problems. Compared with Gaussian random initialization, the use of coarse-grid spectral information generally improves convergence and reduces computational time. Moreover, TLFSI remains competitive with FEAST in terms of both iteration counts and computational cost. These results demonstrate the adaptability of TLFSI to different boundary conditions and further support the use of coarse-grid spectral information for constructing the initial subspace.

To compare PHSI and TLFSI fairly, we consider the computation of the first seventeen nonzero eigenpairs. PHSI directly computes the first seventeen nonzero eigenvalues and eigenvectors, whereas TLFSI is applied to the target interval
$
[0.0,3.0\times10^4],
$
which contains exactly seventeen nonzero eigenvalues. For the pure Neumann problem considered here, the generalized eigenvalues are nonnegative. Therefore, when the left endpoint of the target interval is chosen as \(a=0.0\), the coarse-grid eigenvectors selected by TLFSI-INT coincide with those selected by TLFSI-IDX. Hence, the two initialization strategies produce the same initial subspace in this test, and only TLFSI-INT is reported below.

For both methods, the dimension of the search subspace and the convergence tolerance are chosen consistently. At each outer iteration, PHSI requires the solution of a multiple-right-hand-side linear system of the form
\[
(A+\alpha M)X=B,
\]
whereas TLFSI requires the solution of shifted systems associated with the contour quadrature,
\[
(A-\sigma_j M)X=B.
\]
The shifted systems corresponding to different quadrature nodes are mutually independent and can therefore be solved in parallel. In the present experiments, however, all shifted systems are solved sequentially. Since the shifts in TLFSI are generally complex, each shifted solve is also more expensive than the real-valued linear systems arising in PHSI. Therefore, in addition to the number of outer iterations, we report the total number of matrix solves and the computational time to provide a more comprehensive comparison of the two methods.

\begin{table}[htbp]
\centering
\caption{Comparison of PHSI and TLFSI-INT for the first seventeen nonzero eigenpairs}
\label{tab:phsi_tlfsi_17}
\begingroup
\small
\setlength{\tabcolsep}{8pt}

\begin{tabular}{ccccc}
\toprule
\multirow{2}{*}{$1/h$} &
\multicolumn{2}{c}{PHSI} &
\multicolumn{2}{c}{TLFSI-INT}\\
\cmidrule(lr){2-3}\cmidrule(lr){4-5}
& Iter. (Solves) & Time
& Iter. (Solves) & Time\\
\midrule

12
& 60 (1020) & 6.34
& 3 (238)   & 1.44\\

24
& 51 (867)  & 48.04
& 3 (244)   & 18.96\\

48
& 44 (748)  & 410.25
& 4 (272)   & 265.62\\

\bottomrule
\end{tabular}

\endgroup
\end{table}

Table~\ref{tab:phsi_tlfsi_17} shows that TLFSI-INT requires substantially fewer outer iterations than PHSI for all three mesh sizes. For \(1/h=12\) and \(24\), TLFSI-INT converges in only \(3\) outer iterations, compared with \(60\) and \(51\) iterations for PHSI, respectively. For the finest mesh, \(1/h=48\), TLFSI-INT requires \(4\) iterations, whereas PHSI requires \(44\) iterations.

The reduction in the number of outer iterations also leads to a substantial decrease in the total number of matrix solves. For \(1/h=12\), the number of matrix solves decreases from \(1020\) for PHSI to \(238\) for TLFSI-INT. For \(1/h=24\) and \(48\), the corresponding numbers decrease from \(867\) to \(244\) and from \(748\) to \(272\), respectively. Consequently, TLFSI-INT is faster than PHSI on all three meshes. In particular, for \(1/h=48\), the computational time is reduced from \(410.25\) seconds for PHSI to \(265.62\) seconds for TLFSI-INT.

Although the shifted systems in TLFSI are generally more expensive than the real-valued systems in PHSI, the significantly smaller number of outer iterations and matrix solves compensates for this additional cost in the present experiments. Moreover, since the shifted systems associated with different quadrature nodes are independent, the computational cost of TLFSI can potentially be further reduced through parallel solution of these systems. Overall, the results demonstrate that TLFSI-INT provides substantially faster convergence than PHSI for this eigenvalue computation task while also reducing the total computational cost.

% ============================================================
%  7. CONCLUSIONS
% ============================================================
\section{Conclusions}\label{sec:conclusions}

We have proposed a two-level filtered subspace iterative (TLFSI)
method for interior eigenvalue problems arising from finite element
discretizations of self-adjoint PDEs. Its initialization strategy
adaptively determines the coarse mesh size $H$ according to the target
spectral interval and identifies a coarse eigenspace approximating the
target invariant subspace.

Under the assumptions of the analysis, we have established an
$\mathcal O(H^2)$ bound for the initial subspace gap and proved
geometric convergence of the subspace iteration with exact filtering.
The contraction factor $\rho_h$ depends only on the filter separation
between the spectral region of interest and the remainder of the
spectrum, while the coarse mesh size $H$ controls the initial subspace
error. We have also derived error estimates for the Ritz vectors and Ritz values
under the corresponding separation conditions.

Numerical experiments for three-dimensional linear elasticity have
shown that the proposed initialization provides more accurate
approximations to the target invariant subspace than random
initialization. In comparisons with FEAST and PHSI, TLFSI has achieved
reductions in iteration counts and computational time, with the gains
depending on the target interval and filter parameters.

\section*{Funding}
This work was supported by the National Natural Science Foundation of
China [Grant No. 12371373, 12671450]; the Postgraduate Scientific
Research Innovation Project of Xiangtan University [Grant No. 
XDCX2025Y209]; and the Postgraduate Scientific Research Innovation
Project of Hunan Province [Grant No. CX20250935].

\bibliographystyle{unsrt}
\bibliography{references}

@Article{LiXu23,
  author  = {Qigang Liang and Xuejun Xu},
  title   = {A Two-Level Preconditioned {Helmholtz} Subspace Iterative Method for {Maxwell} Eigenvalue Problems},
  journal = {SIAM Journal on Numerical Analysis},
  volume  = {61},
  number  = {2},
  pages   = {642--674},
  year    = {2023},
  doi     = {10.1137/21M1392012},
}

@Article{LiXu26,
  author        = {Qigang Liang and Xuejun Xu},
  title         = {A Two-Level Additive {Schwarz} Method for Computing Interior Multiple and Clustered Eigenvalues of Symmetric Elliptic Operators},
  journal       = {arXiv preprint arXiv:2604.13889},
  year          = {2026},
  eprint        = {2604.13889},
  archivePrefix = {arXiv},
  primaryClass  = {math.NA},
}

@article{liang2024,
  title={A Two-Level Block Preconditioned Jacobi--Davidson Method for Multiple and Clustered Eigenvalues of Elliptic Operators},
  author={Liang, Qigang and Wang, Wei and Xu, Xuejun},
  journal={SIAM Journal on Numerical Analysis},
  volume={62},
  number={2},
  pages={998--1019},
  year={2024},
  publisher={SIAM}
}

@Article{LiXu21,
  author  = {Qigang Liang and Xuejun Xu},
  title   = {A Two-Level Preconditioned {Helmholtz--Jacobi--Davidson}
             Method for the {Maxwell} Eigenvalue Problem},
  journal = {Journal of Computational Physics},
  volume  = {443},
  pages   = {110564},
  year    = {2021},
  doi     = {10.1016/j.jcp.2021.110564},
}

@Article{WaXu21,
  author  = {Wei Wang and Xuejun Xu},
  title   = {A Two-Level Overlapping {Schwarz} Method with Subspace
             Correction for the {Stokes} Eigenvalue Problem},
  journal = {Journal of Scientific Computing},
  volume  = {86},
  number  = {2},
  pages   = {1--23},
  year    = {2021},
  doi     = {10.1007/s10915-020-01387-x},
}

@Article{Po09,
  author  = {Eric Polizzi},
  title   = {Density-Matrix-Based Algorithm for Solving Eigenvalue
             Problems},
  journal = {Physical Review B},
  volume  = {79},
  number  = {11},
  pages   = {115112},
  year    = {2009},
  doi     = {10.1103/PhysRevB.79.115112},
}

@Article{TaPo14,
  author  = {Ping Tak Peter Tang and Eric Polizzi},
  title   = {{FEAST} as a Subspace Iteration Eigensolver Accelerated by Approximate Spectral Projection},
  journal = {SIAM Journal on Matrix Analysis and Applications},
  volume  = {35},
  number  = {2},
  pages   = {354--390},
  year    = {2014},
  doi     = {10.1137/13090866X},
}

@Article{GuPoTaVi15,
  author  = {Stefan Guettel and Eric Polizzi and Ping Tak Peter Tang and Gautier Viaud},
  title   = {Zolotarev Quadrature Rules and Load Balancing for the {FEAST} Eigensolver},
  journal = {SIAM Journal on Scientific Computing},
  volume  = {37},
  number  = {4},
  pages   = {A2100--A2122},
  year    = {2015},
  doi     = {10.1137/140980090},
}

@Article{AuTr15,
  author  = {Anthony P. Austin and Lloyd N. Trefethen},
  title   = {Computing Eigenvalues of Real Symmetric Matrices with
             Rational Filters in Real Arithmetic},
  journal = {SIAM Journal on Scientific Computing},
  volume  = {37},
  number  = {3},
  pages   = {A1365--A1387},
  year    = {2015},
  doi     = {10.1137/140984129},
}

@Article{Bornemann11,
  author  = {Folkmar Bornemann},
  title   = {Accuracy and Stability of Computing High-Order Derivatives of
             Analytic Functions by {Cauchy} Integrals},
  journal = {Foundations of Computational Mathematics},
  volume  = {11},
  number  = {1},
  pages   = {1--63},
  year    = {2011},
  doi     = {10.1007/s10208-010-9075-z},
}

@Article{LiSaYa16,
  author  = {Lin Lin and Yousef Saad and Chao Yang},
  title   = {Approximating Spectral Densities of Large Matrices},
  journal = {SIAM Review},
  volume  = {58},
  number  = {1},
  pages   = {34--65},
  year    = {2016},
  doi     = {10.1137/130934283},
}

@Book{PaSa82,
  author    = {Beresford N. Parlett},
  title     = {The Symmetric Eigenvalue Problem},
  publisher = {SIAM},
  address   = {Philadelphia},
  year      = {1998},
  edition   = {2nd},
  doi       = {10.1137/1.9781611971163},
}

@Book{BaJePl05,
  author    = {Zhaojun Bai and James Demmel and Jack Dongarra and
               Axel Ruhe and Henk van~der~Vorst},
  title     = {Templates for the Solution of Algebraic Eigenvalue
               Problems: A Practical Guide},
  publisher = {SIAM},
  address   = {Philadelphia},
  year      = {2000},
  doi       = {10.1137/1.9780898719581},
}

@Article{SlVo96,
  author  = {Gerard L. G. Sleijpen and Henk A. Van~der~Vorst},
  title   = {A {Jacobi--Davidson} Iteration Method for Linear
             Eigenvalue Problems},
  journal = {SIAM Journal on Matrix Analysis and Applications},
  volume  = {17},
  number  = {2},
  pages   = {401--425},
  year    = {1996},
  doi     = {10.1137/S0895479894270427},
}

@Article{ZhHwCa07,
  author  = {Weichung Zhao and Feng-Nan Hwang and Xiao-Chuan Cai},
  title   = {Parallel Two-Level Domain Decomposition Based
             {Jacobi--Davidson} Algorithms for Computing a Few
             Low-Frequency Modes of Large Sparse Generalized
             Eigenvalue Problems},
  journal = {Computer Methods in Applied Mechanics and Engineering},
  volume  = {196},
  number  = {8},
  pages   = {1471--1490},
  year    = {2007},
  doi     = {10.1016/j.cma.2006.09.016},
}

@Book{GoVa13,
  author    = {Gene H. Golub and Charles F. Van~Loan},
  title     = {Matrix Computations},
  publisher = {The Johns Hopkins University Press},
  address   = {Baltimore},
  year      = {2013},
  edition   = {4th},
}

@article{xi2016computing,
  title={Computing partial spectra with least-squares rational filters},
  author={Xi, Yuanzhe and Saad, Yousef},
  journal={SIAM Journal on Scientific Computing},
  volume={38},
  number={5},
  pages={A3020--A3045},
  year={2016},
  publisher={SIAM}
}

@article{wang2026shifted,
  title={A shifted Laplace rational filter for large-scale eigenvalue problems},
  author={Wang, Biyi and Meerbergen, Karl and Vandebril, Raf and An, Hengbin and Mo, Zeyao},
  journal={Journal of Computational and Applied Mathematics},
  pages={117544},
  year={2026},
  publisher={Elsevier}
}

@book{saad2011numerical,
  title={Numerical methods for large eigenvalue problems: revised edition},
  author={Saad, Yousef},
  year={2011},
  publisher={SIAM}
}

@article{sakurai2003projection,
  title={A projection method for generalized eigenvalue problems using numerical integration},
  author={Sakurai, Tetsuya and Sugiura, Hiroshi},
  journal={Journal of computational and applied mathematics},
  volume={159},
  number={1},
  pages={119--128},
  year={2003},
  publisher={Elsevier}
}

@article{li2016thick,
  title={A thick-restart {L}anczos algorithm with polynomial filtering for {H}ermitian eigenvalue problems},
  author={Li, Ruipeng and Xi, Yuanzhe and Vecharynski, Eugene and Yang, Chao and Saad, Yousef},
  journal={SIAM Journal on Scientific Computing},
  volume={38},
  number={4},
  pages={A2512--A2534},
  year={2016},
  publisher={SIAM}
}

@inproceedings{ni2026adapolysi,
  title={AdaPolySI: Adaptive Polynomial Filtered Subspace Iteration for Hermitian Interior Eigenvalue Problems},
  author={Ni, Yuhui and Xu, Xiaofei and Li, Shengguo and Zhang, Juan and Chen, Juan and Wang, Jianchun and Roman, Jose E},
  booktitle={Proceedings of the 40th ACM International Conference on Supercomputing},
  pages={1259--1271},
  year={2026}
}

@book{golub2013matrix,
  title={Matrix computations},
  author={Golub, Gene H and Van Loan, Charles F},
  year={2013},
  publisher={JHU press}
}

@article{di2016efficient,
  title={Efficient estimation of eigenvalue counts in an interval},
  author={Di Napoli, Edoardo and Polizzi, Eric and Saad, Yousef},
  journal={Numerical Linear Algebra with Applications},
  volume={23},
  number={4},
  pages={674--692},
  year={2016},
  publisher={Wiley Online Library}
}

@article{xi2018fast,
  title={Fast computation of spectral densities for generalized eigenvalue problems},
  author={Xi, Yuanzhe and Li, Ruipeng and Saad, Yousef},
  journal={SIAM Journal on Scientific Computing},
  volume={40},
  number={4},
  pages={A2749--A2773},
  year={2018},
  publisher={SIAM}
}

@article{lee2018,
  author = {Lee, Seungwoo and Kwak, Do Young and Sim, Imbo},
  title = {Immersed Finite Element Method for Eigenvalue Problems in Elasticity},
  journal = {Advances in Applied Mathematics and Mechanics},
  year = {2018},
  volume = {10},
  number = {2},
  pages = {424-444},
  doi = {10.4208/aamm.OA-2016-0189},
  url = {https://doi.org/10.4208/aamm.OA-2016-0189}
}

@article{geevers2018dispersion,
  title={Dispersion properties of explicit finite element methods for wave propagation modelling on tetrahedral meshes},
  author={Geevers, Sjoerd and Mulder, Wim A and van der Vegt, Jaap JW},
  journal={Journal of scientific computing},
  volume={77},
  number={1},
  pages={372--396},
  year={2018},
  publisher={Springer}
}

@book{achenbach2012wave,
  title={Wave propagation in elastic solids},
  author={Achenbach, Jan},
  year={2012},
  publisher={Elsevier}
}

@article{xie2014multigrid,
  title={A multigrid method for eigenvalue problem},
  author={Xie, Hehu},
  journal={Journal of Computational Physics},
  volume={274},
  pages={550--561},
  year={2014},
  publisher={Elsevier}
}

@article{xu2026adaptive,
  title={Adaptive Polynomial Filtering for Hermitian Interior Eigenproblems: Convergence Analysis},
  author={Xu, Xiaofei and Ni, Yuhui and Li, Shengguo and Zhang, Juan},
  journal={arXiv preprint arXiv:2604.00914},
  year={2026}
}

@book{lehoucq1998arpack,
  title={ARPACK users' guide: solution of large-scale eigenvalue problems with implicitly restarted Arnoldi methods},
  author={Lehoucq, Richard B and Sorensen, Danny C and Yang, Chao},
  year={1998},
  publisher={SIAM}
}

@article{polizzi2020feast,
  title={{FEAST} eigenvalue solver v4. 0 user guide},
  author={Polizzi, Eric},
  journal={arXiv preprint arXiv:2002.04807},
  year={2020}
}

@article{amestoy2001fully,
  title={A fully asynchronous multifrontal solver using distributed dynamic scheduling},
  author={Amestoy, Patrick R and Duff, Iain S and L'Excellent, Jean-Yves and Koster, Jacko},
  journal={SIAM Journal on Matrix Analysis and Applications},
  volume={23},
  number={1},
  pages={15--41},
  year={2001},
  publisher={SIAM}
}
\end{document}